\documentclass[bj,authoryear]{imsart}

\usepackage[section]{algorithm}
\usepackage{algorithmic}

\startlocaldefs
\theoremstyle{plain}
\newtheorem{thm}{Theorem}[section]
\newtheorem{lem}[thm]{Lemma}
\newtheorem{cor}[thm]{Corollary}
\theoremstyle{remark}
\newtheorem{rem}[thm]{Remark}
\newtheorem{example}[thm]{Example}
\newtheorem{defi}[thm]{Definition}

\newcommand{\inI}{I^{\circ}}
\newcommand{\dint}{{\rm d}}
\newcommand{\N}{{\mathbb N}}

\endlocaldefs

\begin{document}

\begin{frontmatter}
\title{Reversibility of elliptical slice sampling revisited}
\runtitle{Reversibility of elliptical slice sampling revisited}

\begin{aug}
\author[A]{\inits{M.H.}\fnms{Mareike		}~\snm{Hasenpflug}\ead[label=e1]{mareike.hasenplfug@uni-passau.de}}
\author[B]{\inits{V.T.}\fnms{Viacheslav}~\snm{Telezhnikov}\ead[label=e2]{viacheslav.telezhnikov@expertanalytics.de}}
\author[A]{\inits{D.R.}\fnms{Daniel}~\snm{Rudolf}\ead[label=e3]{daniel.rudolf@uni-passau.de}}
\address[A]{Faculty of Computer Science and Mathematics, Universit\"at Passau, Innstra\ss e 33, 94032 Passau\printead[presep={,\ }]{e1,e3}}

\address[B]{previously known as Viacheslav Natarovskii, Expert~Analytics~GmbH, Hubertusstra\ss e~83, 82131 Gauting\printead[presep={,\ }]{e2}}
\end{aug}

\begin{abstract}
We extend elliptical slice sampling, a Markov chain transition kernel suggested in \cite{MuAdMa10}, to infinite-dimensional separable Hilbert spaces and discuss its well-definedness.
We point to a regularity requirement, provide an alternative proof of the desirable reversibility property and show that it induces a positive semi-definite Markov operator. Crucial within the proof of the formerly mentioned results is the analysis of a shrinkage Markov chain that may be interesting on its own.
\end{abstract}

\begin{keyword}
\kwd{Elliptical slice sampling}
\kwd{reversibility}
\kwd{shrinkage procedure}
\end{keyword}

\end{frontmatter}


\section{Introduction}
\label{sec:introduction}

Markov chain Monte Carlo simulations are one of the major tools for approximate sampling of posterior distributions in the context of Bayesian inference.
Elliptical slice sampling (ESS), which has been proposed in \cite{MuAdMa10}, provides a popular  algorithmic transition mechanism, see e.g. \cite{nishihara2014parallel,murray2016pseudo,lie2021dimension}, that leads to a Markov chain which suits the goal of approximate sampling.

The idea of ESS is based on a particular Gaussian Metropolis random walk, see \cite{neal1999regression}, that is nowadays sometimes called preconditioned Crank-Nicolson Metropolis (see \cite{cotter2013mcmc, rudolf2018generalization}), and the shrinkage procedure also due to Neal, see \cite{neal2003slice}. Given the current state, randomly a suitable acceptance region (a level set) and an ellipse is specified. Afterwards, by using the aforementioned shrinkage procedure, the next instance of the Markov chain is generated on the intersection of the ellipse and the acceptance region. 
		
Computationally of advantage is that 
this intersection is a one-dimensional compact set, where its size varies in a way that the jump distances are in each iteration `self-adapted' to the current location and target distribution. Moreover, those customized `probing regions' guarantee a well adjusted exploration of the state space. Therefore, appreciated features of ESS compared to the Gaussian Metropolis random walk are that there are no rejections, that there is no required tuning of a step-size parameter and that it allows for larger jumps, since a richer choice of possible updates is available, cf. \cite{MuAdMa10}.

From the theory side, recently in \cite{pmlr-v139-natarovskii21a} geometric convergence on finite-dimensional spaces has been proven under weak assumptions on the target distribution.
Moreover, numerical experiments in \cite{MuAdMa10,pmlr-v139-natarovskii21a, gessner2022integrals, ray2020efficient, grenioux23a} indicate dimension independent performance of ESS. That motivates the question of well-defined\-ness
on (possibly) infinite-dimen\-sional separable Hilbert spaces. 
Generally, almost always
the first step in investigating Markov chain sampling  (in new scenarios) is the verification of the correct desired stationary distribution. This is fundamental 
because all convergence results, ranging from plain ergodicity (see e.g.\cite{Ti94}), over convergence in total variation distance (cf. \cite{meyn2009markov}) and Wasserstein distance (cf. \cite[Section 20]{douc2018markov}), to spectral gaps with implications for time averages (see e.g. \cite{Jo04,rudolf2012explicit}), 
rely on the target being the stationary distribution. Therefore we focus on the verification of reversibilty, that implies stationarity and (quite commonly) is the base upon which any further convergence analyses are built on.
Our contribution regarding ESS is fourfold:
\begin{enumerate}
	\item The algorithm contains a `shrinkage loop' and we provide a sufficient condition on the distribution of interest for the termination of that loop, which leads to the well-definedness of the transition mechanism and the corresponding Markov chain.
	\item We illuminate that the process on the ellipse actually relies on a Markov chain on $[0,2\pi)$ that is reversible~w.r.t.~the uniform distribution on a suitably transformed acceptance region.  
	\item
	By providing alternative arguments to \cite[Section~2.3]{MuAdMa10} we prove reversibility of the ESS transition kernel in infinite-dimensional settings, which particularly implies that the target measure is a stationary distribution. 
	\item We show that the Markov operator induced by the ESS transition kernel is positive semi-definite on the space of functions that are square integrable with respect to (w.r.t.) the target distribution.
\end{enumerate}
In contrast to the finite-dimensional framework,
for which ESS has been proposed, we consider a possibly infinite-dimensional scenario. 

	For such settings, beyond
	the already mentioned preconditioned Crank-Nicolson Metropolis, including its various modifications (see \cite{cotter2013mcmc, law2014proposal, pinski2015algorithms, rudolf2018generalization}), 
	there exist a variety of other Markov chain Monte Carlo approaches. 
	For example, the Metropolis adjusted Langevin algorithm (see \cite{cotter2013mcmc}), Hamiltonian Monte Carlo sampling (see \cite{beskos2011hybrid, ottobre2016function}) and methods that incorporate the geometry of the posterior distribution (see \cite{cui2016dimension, beskos2017geometric}). Beneficial within the ESS approach may be the fact that no gradient knowledge and no tuning of a step-size parameter is required.

The distribution of interest is specified on an infinite-dimensional separable Hilbert space $\mathcal{H}$ which is equipped with its corresponding Borel $\sigma$-algebra $\mathcal{B}(\mathcal{H})$. We will see 
that ESS is well-defined in such a framework, as already suggested in \cite{lie2021dimension}. We consider $\rho: \mathcal{H} \rightarrow (0, \infty)$ as the likelihood function and a Gaussian reference measure $\mu_0 = \mathcal{N}(0,C)$ defined on $\mathcal{H}$ as the prior distribution, where $C\colon \mathcal{H} \to \mathcal{H}$ is a non-singular covariance operator\footnote{This means that $C\colon \mathcal{H}\to \mathcal{H}$ is a linear bounded, self-adjoint and positive definite trace class operator with ${\rm ker}\, C = \{0\}$.}. 
Then, the probability measure of interest, the posterior/target distribution, denoted by $\mu$, is given as 
\[
\mu({\rm d}x) = \frac{1}{Z}\, \rho(x) \mu_0({\rm d}x)
\]
with normalizing constant $Z=\int_{\mathcal{H}} \rho(x) \mu_0({\rm d}x)$. We consider in the following ESS for approximate sampling of $\mu$ and, in particular, show that if $\rho$ is lower-semicontinuous, i.e., the super level sets of $\rho$ are open sets (within $\mathcal{H}$), then the while-loop in the shrinkage procedure terminates and the transition mechanism leads to a transition kernel that is reversible w.r.t. $\mu$.

We briefly outline the structure of the paper. At the beginning of Section~\ref{sec: prel_and_not} we provide the general setting and the transition mechanisms in algorithmic form. Then, we motivate with a simple example the issue regarding the termination criterion formulated in the algorithms and   
develop a representation of a transition kernel that corresponds to the shrinkage procedure on the circle. In Section~\ref{sec:shrink} we prove that the aforementioned kernel is reversible w.r.t.~a~suitable uniform distribution on a subset of the circle. Finally, in Section~\ref{sec:rev_ess} we show how the reversibility of the shrinkage carries over to the transition kernel of ESS. There
we also obtain the positive semi-definiteness of the corresponding Markov operator.


\section{Preliminaries and notation}
\label{sec: prel_and_not}

We state two equivalent versions of the transition mechanism of
elliptical slice sampling in algorithmic form and provide our
notation. Let $(\Omega,\mathcal{F},\mathbb{P})$ be the underlying probability space of all subsequently used random variables. On the real line $\mathbb{R}$, equipped with its canonical Borel $\sigma$-algebra $\mathcal{B}(\mathbb{R})$, let $\lambda(\cdot)$ denote the Lebesgue measure. For bounded $I\in \mathcal{B}(\mathbb{R})$, with $\lambda(I)>0$ let $\mathcal{U}_I$ be the uniform distribution on $I$. 
In Algorithm~\ref{alg:ess_ala_Murray} the transition mechanism of elliptical slice sampling, as stated in \cite{MuAdMa10},
is presented. 
\begin{algorithm}[htb]
	\caption{Elliptical slice sampling}
	\label{alg:ess_ala_Murray}
	\textbf{Input:}
	$\rho$ and $x_{\text{in}}\in\mathcal{H}$ considered as current state;\\
	\textbf{Output:} $x_{\text{out}}\in\mathcal{H}$ considered as the next state;\\[-2.5ex] 
	\begin{algorithmic}[1]	
		\STATE Draw $T\sim \mathcal{U}_{(0,\rho(x_{\text{in}}))}$, call the result $t$;
		\STATE Draw $W \sim \mu_0 = \mathcal{N}(0, C)$, call the result $w$;
		\STATE Draw $\Gamma\sim \mathcal{U}_{[0,2\pi)}$, call the result $\gamma$;
		\STATE Set $\gamma^{\min} := \gamma - 2\pi$ and set $\gamma^{\max} := \gamma$;
		\WHILE{$\rho(\cos(\gamma) x_{\text{in}} + \sin(\gamma) w) \leq t$}
		\IF{$\gamma < 0$}
		\STATE Set $\gamma^{\min} := \gamma$;
		\ELSE
		\STATE Set $\gamma^{\max} := \gamma$;
		\ENDIF
		\STATE Draw $\Gamma\sim \mathcal{U}_{(\gamma^{\min}, \gamma^{\max})}$, call the result $\gamma$;
		\ENDWHILE
		\RETURN $x_{\text{out}} := \cos(\gamma) x_{\text{in}} + \sin(\gamma) w$.
	\end{algorithmic}
\end{algorithm}
For the analysis below it is convenient to reformulate and split the transition mechanism of Algorithm~\ref{alg:ess_ala_Murray}. 
For this define
for $t\geq 0$ the (super-) level set of $\rho$ w.r.t. $t$ as
\[
\mathcal{H}(t) := \{ x\in \mathcal{H} \colon \rho(x)>t \}
\] 
and for $\alpha,\beta\in [0,2\pi)$ let
\[
I(\alpha,\beta) := 
\begin{cases}
[0,\beta) \cup [\alpha,2\pi), & \alpha > \beta\\
[\alpha,\beta), & \alpha <\beta\\
[0, 2\pi), & \alpha = \beta,		
\end{cases} \qquad
	J(\alpha,\beta) := 
	\begin{cases}
	[0,\beta) \cup [\alpha,2\pi), & \alpha > \beta\\
	[\alpha,\beta), & \alpha <\beta\\
	\emptyset, & \alpha = \beta	
	\end{cases}
\]
be the notation 
for intervals
	that respect the geometry of the circle.  Observe that 
$I(\alpha,\beta)\cap J(\beta,\alpha)=\emptyset$ and $I(\alpha,\beta)\cup J(\beta,\alpha)=[0,2\pi)$. Moreover, for $\alpha\not=\beta$ we have $I(\alpha,\beta)=J(\alpha,\beta)$.
A useful identity is readily available by distinguishing different cases:
\begin{lem}
	\label{lem: I_interval_prop}
	Let $\alpha, \beta, \gamma \in [0,2\pi)$. If $\gamma \neq \alpha$ we have
	$
	\mathbf{1}_{I(\alpha,\beta)}(\gamma) = \mathbf{1}_{J(\beta,\gamma)}(\alpha).
	$
	Conversely if $\gamma \neq \beta$ we have $
	\mathbf{1}_{J(\alpha,\beta)}(\gamma) = \mathbf{1}_{I(\gamma,\alpha)}(\beta).
	$
\end{lem}
For given $x,w\in \mathcal{H}$ define the function $p_{x,w}\colon [0,2\pi) \to \mathcal{H}$ as
\[
p_{x,w}(\theta) := \cos(\theta) x + \sin(\theta) w,
\]
which describes an ellipse in $\mathcal{H}$ with conjugate diameters determined by $x,w$. 
We remind the reader on the definition of the pre-image of $p_{x,w}$, which is, for $A\in\mathcal{B}(\mathcal{H})$ 
given as
\[
p_{x,w}^{-1}(A) := \{ \theta\in [0,2\pi) \colon p_{x,w}(\theta )\in A \}.
\] 
It determines the part of $[0,2\pi)$ that leads via $p_{x,w}$ to elements on the ellipse intersected with $A$.  
In the aforementioned reformulation of Algorithm~\ref{alg:ess_ala_Murray} we aim to highlight the structure of the elliptical slice sampling approach. It is given in Algorithm~\ref{alg:ess_ala_ana}, calling Algorithm~\ref{alg:shrink} as a built-in procedure. This procedure gives a transition mechanism on a set $S\in\mathcal{B}([0,2\pi))$. 
\begin{algorithm}[htb]
	\caption{Shrinkage, called as $\text{shrink}(\theta_{\text{in}},S)$}
	\label{alg:shrink}
	\textbf{Input:} $S\in \mathcal{B}([0,2\pi))$, $\theta_{\text{in}}\in S$ considered as current state;\\
	\textbf{Output:} $\theta_{\text{out}}\in S$ considered as next 
	state; \\[-2.5ex]
	\begin{algorithmic}[1]
		\STATE Set $i := 1$ and draw $\Gamma_i\sim \mathcal{U}_{[0,2\pi)}$, call the result $\gamma_i$;
		\STATE Set $\gamma^{\min}_i := \gamma_i$ and $\gamma^{\max}_i := \gamma_i$;
		\WHILE{$\gamma_i \notin S$}
		\IF{$\gamma_i \in J(\gamma^{\min}_i, \theta_{\text{in}})$}
		\STATE Set $\gamma^{\min}_{i+1} := \gamma_i$ and $\gamma^{\max}_{i+1} := \gamma^{\max}_{i}$;
		\ELSE
		\STATE Set $\gamma^{\min}_{i+1} := \gamma^{\min}_{i}$ and $\gamma^{\max}_{i+1} := \gamma_{i}$;
		\ENDIF
		\STATE Draw $\Gamma_{i+1}\sim \mathcal{U}_{I(\gamma_{i+1}^{\min},\gamma^{\max}_{i+1})}$, call the result $\gamma_{i+1}$;
		\STATE Set $i := i + 1$;
		\ENDWHILE
		\STATE \textbf{return} $\theta_{\text{out}} := \gamma_i$.
	\end{algorithmic}
\end{algorithm}
\begin{algorithm}[htb]
	\caption{Reformulated Elliptical slice sampling}
	\label{alg:ess_ala_ana}
	\textbf{Input:} $\rho$ and $x_{\text{in}}\in\mathcal{H}$ considered as current state;\\
	\textbf{Output:} $x_{\text{out}}\in\mathcal{H}$ considered as 
	next state;\\[-2.5ex] 
	\begin{algorithmic}[1]	
		\STATE Draw $T\sim \mathcal{U}_{(0,\rho(x_{\text{in}}))}$, call the result $t$;
		\STATE Draw $W \sim \mu_0 = \mathcal{N}(0, C)$, call the result $w$;
		\STATE Set $\gamma := \text{shrink}(0,p^{-1}_{x_{\text{in}},w}(\mathcal{H}(t)))$; \qquad (Algorithm~\ref{alg:shrink})
		\RETURN $x_{\text{out}} := \cos(\gamma) x_{\text{in}} + \sin(\gamma) w$.
	\end{algorithmic}
\end{algorithm}

Comparing Algorithm~\ref{alg:ess_ala_Murray} and Algorithm~\ref{alg:ess_ala_ana} one observes that line~1, line~2 and the return-line coincide (after $\gamma$ has been computed). Given realizations $t$ and $w$ line~3 until line~12 of Algorithm~\ref{alg:ess_ala_Murray}, including the while-loop, correspond to calling the shrinkage procedure of Algorithm~\ref{alg:shrink} within Algorithm~\ref{alg:ess_ala_ana} with input $\theta_{\text{in}}=0$ and $S=p_{x_{\text{in}},w}^{-1}(\mathcal{H}(t))$. For convincing yourself that those parts also coincide note that
\[
p_{x_{\text{in}},w}^{-1}(\mathcal{H}(t)) = \{ \theta\in[0,2\pi)\colon \rho(\cos(\theta)x_{\text{in}}+\sin(\theta)w)>t \}
\]
and therefore the termination criterion in the while-loops remains the same. In the first iteration both algorithms draw the angles always uniformly distributed on an interval of length $2\pi$. Moreover, there and in general,
the $2\pi$-periodicity of the function $p_{x_{\text{in}},w}$ is exploited in the construction of the shrinked intervals in the while-loop in Algorithm~\ref{alg:ess_ala_Murray}, whereas in Algorithm~\ref{alg:shrink} we work with the generalized intervals {$I(\alpha,\beta)$ and $J(\alpha,\beta)$} for given $\alpha,\beta\in [0,2\pi)$. 
To finally convince yourself that indeed the same transitions are performed it is useful to specify how one samples uniformly distributed in the generalized intervals. 
Namely, for $\alpha< \beta$, just sample uniformly distributed in $[\alpha,\beta)$ to get a realization w.r.t. $\mathcal{U}_{I(\alpha,\beta)}$.
For $\alpha\geq\beta$ and uniform sampling in $I(\alpha,\beta)$, draw $V\sim\mathcal{U}_{[\alpha-2\pi,\beta)}$ with the result $v$ and set the output as $v+2\pi$ if $v\in [\alpha-2\pi,0)$ and $v$ otherwise.
Employing this procedure for realizing $\mathcal{U}_{I(\alpha,\beta)}$ in Algorithm~\ref{alg:shrink}, driven by the same random numbers as the interval sampling in Algorithm~\ref{alg:ess_ala_Murray}, yields finally the same transitions and angles\footnote{The angles and shrinked intervals coincide up to transformation to $[0,2\pi)$.}.
Notice that the suggested representation of ESS in terms of Algorithm~\ref{alg:shrink} sheds light into a new perspective about the transition mechansim, which is valuable rather for theoretical than practical purposes. For example, taking it literally it is of course computational infeasible to access $S=p^{-1}_{x_{\text{in}}}(\mathcal{H}(t))$, whereas in Algorithm~\ref{alg:ess_ala_Murray} a membership oracle w.r.t $p^{-1}_{x_{\text{in}}}(\mathcal{H}(t))$ is sufficient.

\subsection{
	Properties and notation of the shrinkage procedure
}

The shrinkage procedure of Algorithm~\ref{alg:shrink} (and Algorithm~\ref{alg:ess_ala_Murray}) is only well-defined if the while-loop terminates. In particular, if $\lambda(S)=0$, then for any $I(\alpha,\beta)$, with $\alpha,\beta\in [0,2\pi)$, and $V\sim\mathcal{U}_{I(\alpha,\beta)}$ we have
\begin{equation}
\label{eq: prob_to_be_in_S}
\mathbb{P}(V\in S) = \mathcal{U}_{I(\alpha,\beta)}(S) = \frac{\lambda(S\cap I(\alpha,\beta))}{\lambda(I(\alpha,\beta))} = 0.
\end{equation}
Consequently, for an input $S\in \mathcal{B}([0,2\pi))$ with $\lambda(S)=0$ the shrinkage procedure of Algorithm~\ref{alg:shrink} does not terminate almost surely, since in line~9 there one chooses uniformly distributed in a suitable generalized interval and by \eqref{eq: prob_to_be_in_S} the probability to be in $S$ is zero.
With the following illustrating example we illuminate the well-definedness problem in terms of Algorithm~\ref{alg:ess_ala_ana} with a toy scenario.
\begin{example}
	For $d\in\mathbb{N}$ consider $\mathcal{H}=\mathbb{R}^d$ and let $\varepsilon>0$ as well as  $\mu=\mathcal{N}(0,I)$ be the standard normal distribution in $\mathbb{R}^d$ with $I\in\mathbb{R}^{d\times d}$ being the identity matrix. Moreover, let $\rho \colon \mathbb{R}^d \to [\varepsilon,1+\varepsilon]$ be given as
	\[
	\rho(x) = \mathbf{1}_{[0,1]^d}(x) + \varepsilon,\qquad x\in \mathbb{R}^d.
	\]
	Observe that $\mathcal{H}(t)=[0,1]^d$ for $t> \varepsilon$. We see that the fact that this is a closed set might lead (for certain inputs) to a well-definedness issue. 
	For 

	\[w\in\big\{(\widetilde{w}^{(1)},\dots,\widetilde{w}^{(d)})\in\mathbb{R}^d\colon
	\exists i,j\in \{1,\dots,d\} \;\text{s.t.}\;
	\widetilde{w}^{(i)}<0,\;\widetilde{w}^{(j)}>0
	\big\}\]
	we have
	\begin{align*}
	& p_{0,w}([0,2\pi))  =\{\widetilde{w}\in \mathbb{R}^d\colon \widetilde{w}=s w,\; s\in [-1,1]\},\qquad
	 p_{0,w}([0,2\pi)) \cap \mathcal{H}(t)  = \{0\}\subset \mathbb{R}^d,
	\end{align*}
	such that $p_{0,w}^{-1}(\mathcal{H}(t))=\{0\}\subset [0,2\pi)$. For the random variables $T$ and $W$ as in Algorithm~\ref{alg:ess_ala_ana} we obtain that
	\[
	\mathbb{P}\left(\lambda(p^{-1}_{0,W}(\mathcal{H}(T)))=0\right) = \frac{2^d-2}{2^d(1+\varepsilon)}.
	\]
	Thus, for input $x_{\text{in}}=0$ and $\rho$, with the former probability the while-loop in the shrinkage procedure does not terminate.
\end{example}
In the following we introduce the mathematical objects to formulate sufficient conditions for guaranteeing an almost sure termination of the aforementioned while-loops and a desired reversibility property of the shrinkage procedure.

We start with some notation. For probability measures $\mu,\nu$ defined on possibly different measurable spaces the corresponding product measure on the Cartesian product space is denoted as $\mu\otimes\nu$. Moreover, for the Dirac measure at $v$ (on an arbitrary measurable space) we write $\delta_{v}(\cdot)$. Having two random variables/vectors $X,Y$ we denote the distribution of $X$ as $\mathbb{P}_X$ and the conditional distribution of $X$ given $Y$ as $\mathbb{P}_{X\mid Y}$.

Fix $S\in \mathcal{B}([0,2\pi))$ and let $\theta\in S$ with $\theta_{\text{in}}=\theta$. Define
\begin{align*}
\Lambda &:= \{(\gamma,\gamma^{\min},\gamma^{\max})\in [0,2\pi)^3 
\colon \gamma \in I(\gamma^{\min},\gamma^{\max})\},\\
\Lambda_\theta & := \{ (\gamma,\gamma^{\min},\gamma^{\max})\in [0,2\pi)^3 \colon \gamma,\theta\in I(\gamma^{\min},\gamma^{\max}),  \gamma \neq \theta \}.
\end{align*}
Considering $z_1=(\gamma_1,\gamma_1^{\min},\gamma_1^{\max})$ from Algorithm~\ref{alg:shrink} as a realization of a random vector $Z_1=(\Gamma_1,\Gamma^{\min}_1,\Gamma^{\max}_1)$ on $([0,2\pi)^3,\mathcal{B}([0,2\pi)^3)$, we have by line 1-2 of the aforementioned procedure that the distribution of $Z_1$ is given by
\begin{equation}
\label{eq: init_distr}
\mathbb{P}_{Z_1}(C) = \int_0^{2\pi}
\delta_{(\gamma,\gamma,\gamma)}(C)\,	 \frac{{\rm d}\gamma}{2\pi}, \qquad C\in \mathcal{B}([0,2\pi)^3).
\end{equation}
Assume that $\Theta$ is a random variable
mapping to $S$ with distribution $\mathcal{U}_S$ and consider $\theta\in S$ with $\theta=\theta_{\rm in}$ as a realization of $\Theta$. For convenience we refer to $\Theta$ also as anchor (random variable).
{Given $\Theta=\theta$, note that $\Gamma_1\in I(\Gamma_1^{\min},\Gamma_1^{\max})$, $\theta\in S \subseteq I(\Gamma_1^{\min},\Gamma_1^{\max})$ and $\Gamma_1 \neq \theta$ almost surely, such that $Z_1(\omega)\in \Lambda_\theta$ for almost all $\omega \in \Omega$.}
Moreover, given $\Theta=\theta$ the sequence $(z_n)_{n\in\mathbb{N}}$, with $z_n = (\gamma_n,\gamma_n^{\min},\gamma_n^{\max})\in [0,2\pi)^3$, from iterating over lines 4-9 (ignoring the stopping criterion in the while loop) of Algorithm~\ref{alg:shrink} is a realization of a sequence of random variables $(Z_n)_{n\in\mathbb{N}}$ with $Z_n=(\Gamma_n,\Gamma_n^{\min},\Gamma_n^{\max})$. For illustrative purposes we provide
the dependency graph of $(Z_n)_{n\in\mathbb{N}}$ conditioned on $\Theta=\theta$ in Figure~\ref{fig}.
\begin{figure}[h]
	\begin{center}
		\includegraphics{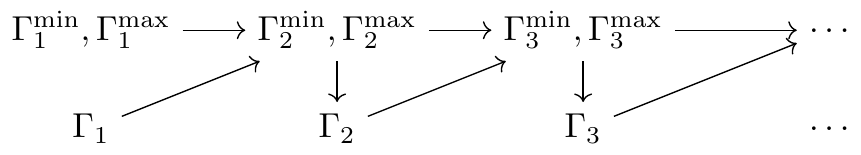}
	\end{center}
	\caption{The dependency graph of the $(Z_n)_{n\in\mathbb{N}}$ conditioned on $\Theta=\theta$.}
	\label{fig}
\end{figure}

Moreover, we call $(Z_n)_{n \in \mathbb{N}}$ the unstopped shrinkage sequence. Conditioned on the anchor $\Theta$ the unstopped shrinkage sequence satisfies the Markov property, i.e.,
	\begin{align}
	\label{al: MP}
	\mathbb{P}_{Z_{n+1} \mid Z_1,\dots,Z_n,\Theta}(C)
	= & \mathbb{P}_{Z_{n+1}\mid Z_n,\Theta}(C), \qquad  C\in \mathcal{B}([0,2\pi)^3).
	\end{align}
From the algorithmic description (cf. Algorithm~\ref{alg:shrink}), for $A\in \mathcal{B}([0,2\pi)), B\in\mathcal{B}([0,2\pi)^2)$, one can read off the following conditional distribution properties
\begin{align}
\label{al: gamma_given_interval}
&\mathbb{P}_{\Gamma_{n+1}\mid \Gamma_{n+1}^{\min},\Gamma_{n+1}^{\max},{Z_n,\dots,Z_1,} \Theta} (A)  =
\mathbb{P}_{\Gamma_{n+1}\mid \Gamma_{n+1}^{\min},\Gamma_{n+1}^{\max}} (A)
= \mathcal{U}_{I(\Gamma_{n+1}^{\min},\Gamma_{n+1}^{\max})}(A),\\
\label{al: interval_given_prev} 
&\mathbb{P}_{\Gamma_{n+1}^{\min},\Gamma_{n+1}^{\max}\mid Z_{n},\Theta}(B)  = 
\mathbf{1}_{J(\Gamma_{n}^{\min},\Theta)}(\Gamma_{n}) 
\delta_{(\Gamma_{n},\Gamma_{n}^{\max})}(B) + 
\mathbf{1}_{I(\Theta,\Gamma_{n}^{\max})}(\Gamma_{n}) 
\delta_{(\Gamma_{n}^{\min},\Gamma_{n})}(B).
\end{align}  
From \eqref{al: gamma_given_interval} 
we deduce $\mathbb{P}_{\Gamma_{n+1}\mid \Gamma_{n+1}^{\min},\Gamma_{n+1}^{\max},Z_n,\dots,Z_1,\Theta}=\mathbb{P}_{\Gamma_{n+1}\mid \Gamma_{n+1}^{\min},\Gamma_{n+1}^{\max},Z_n,\Theta}$, such that by \eqref{al: interval_given_prev} the right-hand side of \eqref{al: MP} can be represented as
\begin{align*}
\mathbb{P}_{Z_{n+1}\mid Z_n,\Theta} (A\times B) & = \int_B \mathbb{P}_{\Gamma_{n+1}\mid \Gamma_{n+1}^{\min} = \gamma^{\min} ,\Gamma_{n+1}^{\max}=\gamma^{\max},Z_n, \Theta} (A) \, \mathbb{P}_{\Gamma_{n+1}^{\min},\Gamma_{n+1}^{\max}\mid Z_{n},\Theta}({\rm d}\gamma^{\min} {\rm d} \gamma^{\max})\\
& = \mathbf{1}_{J(\Gamma_{n}^{\min},\Theta)}(\Gamma_{n}) \int_B \mathcal{U}_{I(\gamma^{\min},\gamma^{\max})}(A) \delta_{(\Gamma_{n},\Gamma_{n}^{\max})}({\rm d}\gamma^{\min} {\rm d} \gamma^{\max}) \\
&  \qquad  +
\mathbf{1}_{I(\Theta,\Gamma_n^{\max})}(\Gamma_n)  \int_B \mathcal{U}_{I(\gamma^{\min},\gamma^{\max})}(A)
\delta_{(\Gamma_n^{\min},\Gamma_n)}({\rm d} \gamma^{\min}{\rm d}\gamma^{\max}).
\end{align*}
We can rewrite this in terms of a transition kernel. Given $\Theta=\theta$ and current state 
$z=(\gamma,\gamma^{\min},\gamma^{\max})\in \Lambda_\theta$ we define a transition kernel $R_\theta$ on $\Lambda_\theta \times \mathcal{B}([0,2\pi)^3)$ by
\begin{align*}
& R_{\theta}((\gamma,\gamma^{\min},\gamma^{\max}), C) \;:= \mathbb{P}_{Z_{n+1}\mid Z_n=(\gamma,\gamma^{\min},\gamma^{\max}),\Theta=\theta} (C) \\
& =\mathbf{1}_{J(\gamma^{\min},\theta)}(\gamma) \int_C \mathcal{U}_{I(\alpha^{\min},\alpha^{\max})}({\rm d}\alpha) \delta_{\gamma}({\rm d} \alpha^{\min}) \delta_{\gamma^{\max}}({\rm d} \alpha^{\max}) \\
&  \qquad  +
\mathbf{1}_{I(\theta,\gamma^{\max})}(\gamma)  \int_C \mathcal{U}_{I(\alpha^{\min},\alpha^{\max})}({\rm d}\alpha)
\delta_{\gamma^{\min}}({\rm d} \alpha^{\min}) \delta_{\gamma}({\rm d}\alpha^{\max}), \quad\; C\in \mathcal{B}([0,2\pi)^3).
\end{align*}
We note the following properties:
\begin{lem}  \label{lem: repres_R_theta}

	For $\theta\in S$, for any $z=(\gamma,\gamma^{\min},\gamma^{\max})\in \Lambda_\theta$ and any $C\in \mathcal{B}([0,2\pi)^3)$
		the transition kernel $R_\theta$ of the unstopped shrinkage sequence conditioned on the anchor $\Theta=\theta$ satisfies 
		\begin{align*}
		& R_\theta((\gamma,\gamma^{\min},\gamma^{\max}), C) \\
		& = \mathbf{1}_{I(\gamma,\gamma^{\min})}(\theta) \cdot \mathcal{U}_{I(\gamma,\gamma^{\max})} \otimes 
		\delta_{(\gamma,\gamma^{\max})}(C) + \mathbf{1}_{J(\gamma^{\max},\gamma)}(\theta) \cdot \mathcal{U}_{I(\gamma^{\min},\gamma)} \otimes 
		\delta_{(\gamma^{\min},\gamma)}(C)\\
		& = \mathbf{1}_{I(\gamma,\gamma^{\max})}(\theta) \cdot \mathcal{U}_{I(\gamma,\gamma^{\max})} \otimes 
		\delta_{(\gamma,\gamma^{\max})}(C) + \mathbf{1}_{J(\gamma^{\min},\gamma)}(\theta) \cdot \mathcal{U}_{I(\gamma^{\min},\gamma)} \otimes 
		\delta_{(\gamma^{\min},\gamma)}(C),
		\end{align*}
		which reflects the case distinction of Algorithm~\ref{alg:shrink}.
		Moreover, we have $R_\theta((\gamma,\gamma^{\min},\gamma^{\max}),\Lambda_\theta)=1$.
\end{lem}
\begin{proof}
	The first equality follows by Lemma~\ref{lem: I_interval_prop} and the second equality by taking $z\in \Lambda_\theta$, in particular, $\theta\in I(\gamma^{\min},\gamma^{\max})$ into account. 
	It remains to show $R_\theta((\gamma,\gamma^{\min},\gamma^{\max}),\Lambda_\theta)=1$.
	Note that
	\begin{align*}
	\mathcal{U}_{I(\gamma,\gamma^{\max})} \otimes 
	\delta_{(\gamma,\gamma^{\max})}(\Lambda_\theta)  
	=  \mathbf{1}_{I(\gamma,\gamma^{\max})}(\theta) \cdot \mathcal{U}_{I(\gamma,\gamma^{\max})} \left(I(\gamma,\gamma^{\max})\setminus\{\theta\}\right)
	= \mathbf{1}_{I(\gamma,\gamma^{\max})}(\theta) 
	\end{align*}
	and by the same arguments
	$
	\mathcal{U}_{I(\gamma^{\min},\gamma)} \otimes 
	\delta_{(\gamma^{\min},\gamma)}(\Lambda_\theta) = 	\mathbf{1}_{I(\gamma^{\min},\gamma)}(\theta).
	$
	Since $\gamma\in I(\gamma^{\min},\gamma^{\max})$ we have 
	\[
	I(\gamma^{\min},\gamma^{\max}) = J(\gamma^{\min},\gamma)\cup I(\gamma,\gamma^{\max})
	\quad \text{and} \quad 
	J(\gamma^{\min},\gamma)\cap I(\gamma,\gamma^{\max})=\emptyset,
	\]
	such that 
	\begin{align*}
	R_\theta(z, \Lambda_\theta) & = \mathbf{1}_{I(\gamma,\gamma^{\max})}(\theta)  + \mathbf{1}_{J(\gamma^{\min},\gamma) \cap I(\gamma^{\min},\gamma) }(\theta)
	= \mathbf{1}_{I(\gamma,\gamma^{\max})}(\theta)  + \mathbf{1}_{J(\gamma^{\min},\gamma)  }(\theta) = 1.
	\qedhere
	\end{align*}
\end{proof}
A useful representation of the transition kernel in terms of random variables follows readily from the previous lemma.
\begin{lem} 
	\label{lem: kernel_as_expect}
	For any $\theta\in S$, for any $z\in \Lambda_\theta$, any $C\in \mathcal{B}([0,2\pi)^3)$ and any $n\geq 2$ the transition kernel $R_\theta$ of the unstopped shrinkage sequence allows the representation as conditional expectation of the previous iterates and the anchor $\Theta=\theta$ as
	\begin{align}
	\label{eq: kernel_as_expect}
	& R_\theta(z, C) 
	= \mathbb{E}[\mathbf{1}_{I(\Gamma_n^{\min},\Gamma_n^{\max})}(\theta)\, \mathcal{U}_{I(\Gamma_n^{\min},\Gamma_n^{\max})}\otimes \delta_{(\Gamma_n^{\min},\Gamma_n^{\max})}(C)\mid Z_{n-1}=z,\Theta = \theta].
	\end{align}
\end{lem}
We add another property regarding the distribution of $\Theta$ given $Z_1,\dots,Z_n$ that is proven in Appendix~\ref{app: 1}
\begin{lem} \label{lem: theta_conditioned_on_interval}
	For any $n\in\mathbb{N}$ and $A\in \mathcal{B}([0,2\pi))$ the conditional distribution of the anchor $\Theta$ given $Z_1 = (\Gamma_1, \Gamma_1^{\min}, \Gamma_1^{\max}), \ldots, Z_n = (\Gamma_n, \Gamma_n^{\min}, \Gamma_n^{\max})$ of the unstopped shrinkage sequence is
	\begin{equation}
	\label{eq: theta_conditioned_on_interval}
	\mathbb{P}_{\Theta\mid Z_1,\dots, Z_{n}} (A) = \mathbb{P}_{\Theta\mid \Gamma_n^{\min},\Gamma_n^{\max}}(A) = \frac{\mathbb{P}_{\Theta}(A\cap I(\Gamma_n^{\min},\Gamma_n^{\max}))}{\mathbb{P}_\Theta(I(\Gamma_n^{\min},\Gamma_n^{\max}))}.
	\end{equation}
\end{lem}

\subsection{Stopping of the shrinkage procedure}

Now we are aiming to take the stopping criterion within the while loop of Algorithm~\ref{alg:shrink} into account. For this we introduce the $\sigma$-algebras $\mathcal{F}_n := \sigma (Z_1,\dots, Z_n)$ and the natural filtration $\{\mathcal{F}_n\}_{n \in \mathbb{N}}$ of $(Z_n)_{n\in\mathbb{N}}$. We define the (random) termination time $\tau_S$ of the while-loop as the first $n\in\mathbb{N}$ where $\Gamma_n$ is in $S$, i.e.,
\begin{align}
\label{eq:tau-definition}
\tau_S
:=
\inf
\left\{
n \in \mathbb{N}
:
Z_n \in S \times [0, 2\pi)^2
\right\},
\end{align}
where by convention $\inf \emptyset = \infty$. 
Note that $\tau_S$ is a stopping time w.r.t. the natural filtration, since 
\[
\{\tau_S = n\}
=
\bigcap_{k=1}^{n-1}
\left\{
Z_k \notin S \times [0, 2\pi)^2
\right\}
\cap
\left\{
Z_n \in S \times [0, 2\pi)^2
\right\}
\in
\mathcal{F}_n, 
\]
for any $n \in \mathbb{N}$. Moreover, the transition mechanism of Algorithm~\ref{alg:shrink} for input $S$ and $\theta$ can be formulated in terms of a transition kernel if the while-loop conditioned on 
the event
$\Theta=\theta$ terminates almost surely, that is, $\mathbb{P}(\tau_S<\infty\mid \Theta=\theta)=1$. 
	To illuminate a sufficient condition for that property we need the following profane definition.
	\begin{defi}
		A set $S\in \mathcal{B}([0,2\pi))$ is \emph{open on the circle} if for all $\theta \in S$ there exists an $\varepsilon > 0$ such that 
		$I(\theta-\varepsilon \mod 2\pi, \theta+\varepsilon \mod 2\pi) \subseteq S$.
		\end{defi}
	\noindent
	In other words, $S\in \mathcal{B}([0,2\pi))$ is open on the circle if it is the image of an open set under the natural bijection between $[0,2\pi)$ and the one-dimensional sphere.
\begin{lem}
	\label{lem: as_finite_st}  
	Assume that $S\in \mathcal{B}([0,2\pi))$ is
	open on the circle and non-empty. 
	Then, for any $\theta\in S$ we have
	$\mathbb{P}(\tau_S<\infty \mid \Theta=\theta)=1$, i.e., the stopping time, that corresponds to the number of iterations of the while-loop in Algorithm~\ref{alg:shrink} until termination, is almost surely finite.
\end{lem}
\begin{proof}
	By the fact that $S$ is open on the circle and $\theta\in S$ there exists an
	$\varepsilon>0$ such that $I_\theta:=I(\theta_{\varepsilon,-} , \theta_{\varepsilon,+})\subseteq S$, where
	\begin{align*}
	\theta_{\varepsilon,-} & = \theta-\varepsilon \mod 2\pi,\qquad
	\theta_{\varepsilon,+} = \theta+\varepsilon \mod 2\pi,
	\end{align*}
	with $\theta\in I_\theta$.
	Furthermore, note that $\lambda(I_\theta)=2\varepsilon$. Set $\widetilde{S} := S\times [0,2\pi)^2$ and observe that for any  $\gamma^{\min},\gamma^{\max} \in [0,2\pi)$ with $\theta\in I(\gamma^{\min},\gamma^{\max})$ we have
	\begin{align*}
	\mathcal{U}_{I(\gamma^{\min},\gamma^{\max})} (S) 
	& = \frac{\lambda(S\cap I(\gamma^{\min},\gamma^{\max}))}{\lambda(I(\gamma^{\min},\gamma^{\max}))}\\
	&\geq 
	\begin{cases}
	1 & \gamma^{\min}, \gamma^{\max}\in I_{\theta}\\
	\frac{\varepsilon}{\lambda(I(\gamma^{\min},\gamma^{\max}))} & \gamma^{\min} \in I_\theta, \gamma^{\max}\not\in I_\theta \quad \text{or} \quad \gamma^{\min} \not\in I_\theta, \gamma^{\max}\in I_\theta\\
	\frac{2\varepsilon}{\lambda(I(\gamma^{\min},\gamma^{\max}))} & \gamma^{\min}, \gamma^{\max}\not\in I_\theta
	\end{cases}\\
	& \geq \frac{\varepsilon}{2\pi}.
	\end{align*}
	Using this estimate, we obtain for any $z=(\gamma,\gamma^{\min},\gamma^{\max})\in \Lambda_\theta$ that 
	\begin{align*}
	&R_\theta(z,\widetilde{S}) = 
	\mathbf{1}_{I(\gamma,\gamma^{\max})}(\theta) \mathcal{U}_{I(\gamma,\gamma^{\max})}(S)
	+ \mathbf{1}_{J(\gamma^{\min},\gamma)}(\theta) \mathcal{U}_{I(\gamma^{\min},\gamma)}(S)
	\geq \frac{\varepsilon}{2\pi}.
	\end{align*}
	%
	%
	Recall that $Z_1$ with $\mathbb{P}_{Z_1}$ from \eqref{eq: init_distr} satisfies $Z_1\in \Lambda_\theta$ almost surely. Now 
	applying the former estimate iteratively leads to
	\begin{align*}
	& \mathbb{P}(Z_1,\dots,Z_n \not \in \widetilde S^{n}\mid \Theta=\theta) 
	 = \overbrace{\int_{\widetilde{S}^c} \cdots \int_{\widetilde{S}^c}}^{n-1}
	R_\theta(z_{n-1},\widetilde{S}^c) R_\theta(z_{n-2},{\rm d}z_{n-1}) \cdots R_\theta(z_{1},{\rm d}z_{2}) \mathbb{P}_{Z_1}({\rm d}z_1)\\
	& \leq \left(1-\frac{\varepsilon}{2\pi} \right)\; \mathbb{P}(Z_1,\dots,Z_{n-1} \not \in \widetilde S^{n-1} \mid \Theta=\theta) 
	 \leq \dots \leq \left(1-\frac{\varepsilon}{2\pi} \right)^{n-1} \mathbb{P}_{Z_1}(\widetilde{S}^c) \leq \left(1-\frac{\varepsilon}{2\pi} \right)^{n-1},
	\end{align*}
	such that
	\begin{align*}
	\mathbb{P}(\tau_S = \infty \mid \Theta= \theta ) & \leq \lim_{n\to \infty} \mathbb{P}(\tau_S > n \mid \Theta= \theta  ) 
	 \leq \lim_{n\to \infty} \mathbb{P}(Z_1,\dots,Z_n \not \in \widetilde S^{n} \mid \Theta= \theta ) \leq 0
	\end{align*}
	and the proof is finished.
\end{proof}
\begin{cor}
	Assume that the likelihood function/unnormalized density $\rho\colon \mathcal{H} \to (0,\infty)$ on the separable Hilbert space $\mathcal{H}$ is lower semi-continuous, that is, all level sets $\mathcal{H}(t)$ are open sets. Then
	\[
	\mathbb{P}(\tau_{p^{-1}_{x,w}(\mathcal{H}(t))}<\infty \mid \Theta=0) = 1, \qquad  
	\forall x,w\in \mathcal{H} \quad \text{and} \quad t\in (0,\rho(x)),
	\] 
	i.e., the stopping times, that correspond to the number of iterations of the while-loop of Algorithm~\ref{alg:shrink} until termination called within ESS, 
		are almost surely finite.
\end{cor}
\begin{proof}
	By the continuity of $p_{x,w}$, the $2\pi$-periodicity and the fact that $\mathcal{H}(t)$ is open we also have that $p^{-1}_{x,w}(\mathcal{H}(t)) \subseteq [0,2\pi)$ is
	open on the circle
	with $0\in p^{-1}_{x,w}(\mathcal{H}(t))$. Therefore the statement follows by Lemma~\ref{lem: as_finite_st}. 
\end{proof}
The previous corollary tells us that whenever $\rho$ is lower semi-continuous, then calling Algorithm~\ref{alg:shrink} with input $S=p^{-1}_{x,w}(\mathcal{H}(t))$ and $\theta_{\text{in}}=0$ terminates almost surely, such that Algorithm~\ref{alg:ess_ala_ana} also terminates and is well-defined.
\begin{rem}
	Usually the non-termination issue does not seem to have a big influence in applications since most densities of interest have open level sets. For example every continuous density is lower semi-continuous.
	Even if they have single outliers for which the algorithm would not terminate, in practice, the algorithm would shrink and shrink and at some point, because a computer works with machine precision, the shrinked interval cannot be distinguished anymore from the current state such that it will eventually accept and return the current as the next instance. 
	In this sense, the algorithm does not get stuck in an infinite loop.
		However, in that case the algorithmic and mathematical description do not coincide with the implementation. 
\end{rem}

\subsection{Properties of the stopped shrinkage procedure}
\label{sec:shrink}

Now we introduce the stopped random variable $Z_{\tau_S}$ of the Markov chain $(Z_n)_{n\in\mathbb{N}}$. 
For the formal definition on the event $\tau_S = \infty$ use an arbitrary random variable $Z_{\infty}$, that is assumed to be measurable w.r.t. $\mathcal{F}_{\infty} := \sigma (Z_k, k \in \mathbb{N})$. We set 
\[
Z_{\tau_S}(\omega) :=
Z_{\infty}(\omega) \mathbf{1}_{\{\tau_S = \infty\}}(\omega)
+\sum_{k=1}^{\infty}
Z_k(\omega) \mathbf{1}_{\{\tau_S = k\}}(\omega),\qquad \omega\in \Omega.
\]
Notice that $Z_{\tau_S}$ is indeed measurable w.r.t. the $\tau_S$-induced $\sigma$-algebra
\[
\mathcal{F}_{\tau_S}
:=
\{
A \in \mathcal{F}
:
A \cap \{\tau_S = k\}
\in
\mathcal{F}_k, \; k \in \mathbb{N}
\}, 
\]
since for any $A \in \mathcal{F}$ and $k \in \mathbb{N}$ we have
\[
\{Z_{\tau_S} \in A\} \cap \{\tau_S = k\}
=
\{Z_{k} \in A\} \cap \{\tau_S = k\} 
\in
\mathcal{F}_k.
\]
Thus, $Z_{\tau_{S}}= (\Gamma_{\tau_{S}}, \Gamma_{\tau_{S}}^{\min}, \Gamma_{\tau_{S}}^{\max})$ is a $[0,2\pi)^3$-valued random variable and its components $\Gamma_{\tau_{S}}, \Gamma_{\tau_{S}}^{\min}, \Gamma_{\tau_{S}}^{\max}$ are $[0,2\pi)$-valued random variables on the probability space $(\Omega,\mathcal{F}_{\tau_S},\mathbb{P})$. 

Now for given $S\in \mathcal{B}([0,2\pi))$ and an arbitrary $\theta_{\text{in}} \in S$, after the whole construction, we are able to state the transition kernel $Q_S$ on $S \times \mathcal{B}(S)$, also called shrinkage kernel, that corresponds to the transition mechanism of Algorithm~\ref{alg:shrink}. 
For $\theta_{\text{in}}\in S$ and $F\in \mathcal{B}(S)$ it
is given as
\begin{equation}  
Q_S(\theta_{\rm in},F) = \mathbb{P}(\Gamma_{\tau_S}\in F, \tau_S<\infty \mid \Theta=\theta_{\rm in}).
\end{equation}
We formulate the main result regarding the transition kernel $Q_S$ that is essentially used in verifying the reversibility of ESS.

\begin{thm}
	\label{thm: shrink_reversible}
	Let $S\in \mathcal{B}([0,2\pi))$ be 
	open on the circle and non-empty. 
	Then, 
	the shrinkage kernel $Q_S$ is reversible w.r.t. the uniform distribution $\mathcal{U}_S$ on $S$.
\end{thm}
\begin{proof}
	By the Markov property of $(Z_n)_{n\in \mathbb{N}}$ conditioned on $\Theta=\theta$ and Lemma~\ref{lem: kernel_as_expect} we have
	\begin{align*}
	& \mathbb{P}(\Gamma_{\tau_S}\in F, \tau_S<\infty \mid \Theta=\theta) 
	= \sum_{k=1}^\infty \mathbb{P}[\Gamma_{\tau_S}\in F, \tau_S=k \mid \Theta=\theta]\\
	& = \sum_{k=1}^\infty \mathbb{P}[\Gamma_{k}\in F\cap S,\Gamma_1\in S^c,\dots,\Gamma_{k-1}\in S^c \mid \Theta=\theta]
	 = \sum_{k=1}^\infty \mathbb{E}\left[\mathbf{1}_F(\Gamma_{k}) \prod_{i=1}^{k-1} \mathbf{1}_{S^c}(\Gamma_i)  \mid \Theta=\theta\right]\\
	& = \sum_{k=1}^\infty \mathbb{E}\left[\mathbb{E}\left[\mathbf{1}_F(\Gamma_{k}) \prod_{i=1}^{k-1} \mathbf{1}_{S^c}(\Gamma_i) \mid Z_1,\dots,Z_{k-1},\Theta\right]  \mid \Theta=\theta\right]\\
	& = \sum_{k=1}^\infty \mathbb{E}\left[ \prod_{i=1}^{k-1} \mathbf{1}_{S^c}(\Gamma_i) \mathbb{E}\left[\mathbf{1}_F(\Gamma_{k}) \mid Z_1,\dots,Z_{k-1},\Theta\right]  \mid \Theta=\theta\right]\\
	& = \sum_{k=1}^\infty \mathbb{E}\left[ \prod_{i=1}^{k-1} \mathbf{1}_{S^c}(\Gamma_i) \mathbb{E}\left[\mathbf{1}_{F\times [0,2\pi)^2}(Z_{k}) \mid Z_{k-1},\Theta\right]  \mid \Theta=\theta\right]\\
	& 
	\underset{\eqref{eq: kernel_as_expect}}{=} 
	\sum_{k=1}^\infty \mathbb{E}\left[ \prod_{i=1}^{k-1} \mathbf{1}_{S^c}(\Gamma_i) \mathbb{E}\left[\mathbf{1}_{I(\Gamma_k^{\min},\Gamma_k^{\max})}(\theta)\,\mathcal{U}_{I(\Gamma_k^{\min},\Gamma_k^{\max})}(F) \mid Z_{k-1},\Theta\right]  \mid \Theta=\theta\right]\\
	& = \sum_{k=1}^\infty \mathbb{E}\left[ \prod_{i=1}^{k-1} \mathbf{1}_{S^c}(\Gamma_i) \mathbb{E}\left[\mathbf{1}_{I(\Gamma_k^{\min},\Gamma_k^{\max})}(\theta)\,\mathcal{U}_{I(\Gamma_k^{\min},\Gamma_k^{\max})}(F) \mid Z_1,\dots,Z_{k-1},\Theta\right]  \mid \Theta=\theta\right]\\
	& = \sum_{k=1}^\infty \mathbb{E}\left[  \mathbb{E}\left[ \prod_{i=1}^{k-1} \mathbf{1}_{S^c}(\Gamma_i)\mathbf{1}_{I(\Gamma_k^{\min},\Gamma_k^{\max})}(\theta)\,\mathcal{U}_{I(\Gamma_k^{\min},\Gamma_k^{\max})}(F) \mid Z_1,\dots,Z_{k-1},\Theta\right]  \mid \Theta=\theta\right]\\
	& = \sum_{k=1}^\infty \mathbb{E}\left[   \prod_{i=1}^{k-1} \mathbf{1}_{S^c}(\Gamma_i)\mathbf{1}_{I(\Gamma_k^{\min},\Gamma_k^{\max})}(\theta)\,\mathcal{U}_{I(\Gamma_k^{\min},\Gamma_k^{\max})}(F)   \mid \Theta=\theta\right].
	\end{align*}	
	For $F,G\in\mathcal{B}(S)$ this yields by the definition of the conditional distribution
	\begin{align}
	\label{eq: LHS_rev}
	\int_{G} Q_S(\theta,F) \;\mathcal{U}_S({\rm d}\theta)
	=
	\sum_{k=1}^\infty \mathbb{E}\left[  \mathbf{1}_G(\Theta) \prod_{i=1}^{k-1} \mathbf{1}_{S^c}(\Gamma_i)\mathbf{1}_{I(\Gamma_k^{\min},\Gamma_k^{\max})}(\Theta)\,\mathcal{U}_{I(\Gamma_k^{\min},\Gamma_k^{\max})}(F)\right],
	\end{align}
	since $\Theta\sim \mathcal{U}_S$. Note that, by exploiting $F\in\mathcal{B}(S)$ we obtain
	\begin{equation}
	\label{eq: unif_on_I}
	\mathcal{U}_{I(\Gamma_k^{\min},\Gamma_k^{\max})}(F) = \frac{\lambda(F\cap I(\Gamma_k^{\min},\Gamma_k^{\max}))}{\lambda(I(\Gamma_k^{\min},\Gamma_k^{\max}))} 
	= \frac{\mathcal{U}_S(F\cap I(\Gamma_k^{\min},\Gamma_k^{\max}))}{\mathcal{U}_S(I(\Gamma_k^{\min},\Gamma_k^{\max}))}.
	\end{equation}
	Using that and Lemma~\ref{lem: theta_conditioned_on_interval} we modify the expectation within the sum and obtain
	\begin{align}
	\notag
	&\mathbb{E}\left[  \mathbf{1}_G(\Theta) \prod_{i=1}^{k-1} \mathbf{1}_{S^c}(\Gamma_i)\mathbf{1}_{I(\Gamma_k^{\min},\Gamma_k^{\max})}(\Theta)\,\mathcal{U}_{I(\Gamma_k^{\min},\Gamma_k^{\max})}(F)\right]\\
	\notag	 
	= & \mathbb{E}\left[  \prod_{i=1}^{k-1} \mathbf{1}_{S^c}(\Gamma_i)\mathbf{1}_{G\cap I(\Gamma_k^{\min},\Gamma_k^{\max})}(\Theta)\,\frac{\mathcal{U}_S(F\cap I(\Gamma_k^{\min},\Gamma_k^{\max}))}{\mathcal{U}_S(I(\Gamma_k^{\min},\Gamma_k^{\max}))}\right] 
	\\
	\notag
	= & \mathbb{E}\Big[ \mathbb{E} \big[ \prod_{i=1}^{k-1} \mathbf{1}_{S^c}(\Gamma_i)\mathbf{1}_{G\cap
		I(\Gamma_k^{\min},\Gamma_k^{\max})}(\Theta)\,\frac{\mathcal{U}_S(F\cap I(\Gamma_k^{\min},\Gamma_k^{\max}))}{\mathcal{U}_S(I(\Gamma_k^{\min},\Gamma_k^{\max}))}\mid Z_1,\dots,Z_k\big]\Big]\\
	\notag	 
	= & 
	\mathbb{E}\Big[ \prod_{i=1}^{k-1} \mathbf{1}_{S^c}(\Gamma_i)\,\frac{\mathcal{U}_S(F\cap
		I(\Gamma_k^{\min},\Gamma_k^{\max}))}{\mathcal{U}_S(I(\Gamma_k^{\min},\Gamma_k^{\max}))} \mathbb{E} \big[ \mathbf{1}_{G\cap I(\Gamma_k^{\min},\Gamma_k^{\max})}(\Theta) \mid Z_1,\dots,Z_k\big]\Big] \\
	\underset{\eqref{eq: theta_conditioned_on_interval}}{=}
	& 
	\mathbb{E}\Big[ \prod_{i=1}^{k-1} \mathbf{1}_{S^c}(\Gamma_i)\,\frac{\mathcal{U}_S(F\cap
		I(\Gamma_k^{\min},\Gamma_k^{\max}))}{\mathcal{U}_S(I(\Gamma_k^{\min},\Gamma_k^{\max}))} \frac{\mathcal{U}_S(G\cap I(\Gamma_k^{\min},\Gamma_k^{\max})}{\mathcal{U}_S(I(\Gamma_k^{\min},\Gamma_k^{\max})}\Big],
	\label{al: useful_repr_pos_sem}
	\end{align}
	where we also used in the last equation that $\mathbb{P}_\Theta = \mathcal{U}_S$. Now we can reverse the roles of $F$ and $G$, such that arguing backwards leads to
	\[
	\int_G Q_S(\theta,F)\;\mathcal{U}_S({\rm d}\theta )
	= \int_F Q_S(\theta,G)\; \mathcal{U}_S({\rm d}\theta ),
	\]
	which shows the claimed reversibility.
\end{proof}
In the following we show positive semi-definiteness for the Markov operator that corresponds to the transition kernel $Q_S$. We introduce the required objects for drawing this conclusion.
	Define the Hilbert space
	\[
	L_S^2 :=\{f:[0,2\pi) \to \mathbb{R} \colon \int_S f(\theta)^2\, \mathcal{U}_S(\dint \theta) < \infty \}
	\]
	equipped with the inner product $\langle f,g\rangle_S := \int_S f(\theta) g(\theta)\, \mathcal{U}_S(\dint \theta)$ for all $f,g \in L_S^2$.
	The transition kernel $Q_S$ induces a linear operator $\mathrm{Q}_S \colon L_S^2\to L_S^2$, given by
	\begin{equation}\label{eq: def shrinkage operator}
	f \mapsto \left[\alpha \mapsto \int_S f(\theta)\, Q_S(\alpha, \dint\theta)\right], \quad f\in L_S^2.
	\end{equation}
	
	\begin{lem}\label{lem: positive semidefiniteness}
		Assume that $S\in \mathcal{B}([0,2\pi))$ is open on the circle and non-empty.
		Then, the operator $\mathrm{Q}_S$ defined in \eqref{eq: def shrinkage operator}, that is derived from the shrinkage kernel, is positive semi-definite, i.e., for all $f \in L_S^2$ holds $\langle \mathrm{Q}_S f,f\rangle_S \geq 0.$
	\end{lem}
	\begin{proof}
		By \eqref{eq: LHS_rev}, \eqref{eq: unif_on_I} and \eqref{al: useful_repr_pos_sem} we have
		for any $F,G\in\mathcal{B}(S)$ that
		\begin{align*}
		\langle \mathrm{Q}_S \mathbf{1}_F,\mathbf{1}_G \rangle_S 
		= \sum_{k= 1}^\infty \mathbb{E}\Big[ \prod_{i=1}^{k-1} \mathbf{1}_{S^c}(\Gamma_i)\,\mathcal{U}_{I(\Gamma_k^{\min},\Gamma_k^{\max})}(F)\; \mathcal{U}_{I(\Gamma_k^{\min},\Gamma_k^{\max})}(G) \Big].
		\end{align*}
		For fixed $F\in \mathcal{B}(S)$,
		by virtue of the standard machinery of integration theory (exploiting linearity and monotone convergence),
		we extend the former representation from $g=\mathbf{1}_G$ to non-negative $g\in L^2_S$, such that
		\begin{align*}
		\langle \mathrm{Q}_S \mathbf{1}_F,g \rangle_S 
		= \sum_{k= 1}^\infty \mathbb{E}\Big[ \prod_{i=1}^{k-1} \mathbf{1}_{S^c}(\Gamma_i)\,\mathcal{U}_{I(\Gamma_k^{\min},\Gamma_k^{\max})}(F)\; \mathcal{U}_{I(\Gamma_k^{\min},\Gamma_k^{\max})}(g) \Big],
		\end{align*}
		with $\mathcal{U}_{I(\Gamma_k^{\min},\Gamma_k^{\max})}(g):=\int_{I(\Gamma_k^{\min},\Gamma_k^{\max})} g(\theta) \mathcal{U}_{I(\Gamma_k^{\min},\Gamma_k^{\max})}({\rm d} \theta)$. For fixed non-negative $g\in L_S^2$, again by the standard machinery of integration theory, we generalize the latter identity from $f=\mathbf{1}_F$ to non-negative $f\in L_S^2$, such that
		\begin{align*}
		\langle \mathrm{Q}_S f,g \rangle_S 
		= \sum_{k= 1}^\infty \mathbb{E}\Big[ \prod_{i=1}^{k-1} \mathbf{1}_{S^c}(\Gamma_i)\,\mathcal{U}_{I(\Gamma_k^{\min},\Gamma_k^{\max})}(f)\; \mathcal{U}_{I(\Gamma_k^{\min},\Gamma_k^{\max})}(g) \Big].
		\end{align*} 
		Eventually, we also have the latter expression for $f,g\in L_S^2$ by exploiting the bilinearity of the inner product as well as that $f,g$ can be written in terms of the difference of their positive and negative parts. From this last representation it readily follows that $\langle \mathrm{Q}_S f,f\rangle_S   \geq 0$ for all $f \in L_S^2$.
\end{proof}
We finish this section with stating a pushforward invariance property of the transition kernel $Q_S$. 
For general properties regarding pushforward transition kernels we refer to \cite{rudolf2022robust}.
\begin{lem}
	\label{lem: pushforward_invariance}
	Let $S\in \mathcal{B}([0,2\pi))$ be open on the circle and non-empty. For $\theta\in S$ define the function $g_\theta \colon [0,2\pi) \to [0,2\pi)$ by $g_\theta(\alpha) = (\theta - \alpha) \mod 2\pi$. Then, the shrinkage kernel $Q_S$ satisfies
	\[
	Q_{g_\theta^{-1}(S)}(g_\theta^{-1}(\alpha),g_\theta^{-1}(B)) = Q_S(\alpha,B), \qquad \alpha \in S,\, B\in \mathcal{B}(S).
	\]
\end{lem}
\noindent
The proof of the former lemma is shifted to the appendix, see Section~\ref{sec: proof_pushforward_invariance}.
\begin{rem}
	The transition kernel $Q_S$ is not only of interest within ESS. The shrinkage procedure can be used in other Markov chain Monte Carlo methods as well to replace uniform sampling from subsets of a set that can be parameterized by $[0,2\pi)$.
	In fact geodesic slice sampling on the sphere (see \cite{habeck2023geodesic}) and Gibbsian polar slice sampling (see \cite{schaer2023gibbsian}) employ shrinkage to generate points from great circles of a $d-1$-dimensional Euclidean unit sphere.
	Of course in such scenarios one can exploit all properties of the shrinkage kernel $Q_S$ established in this section.
\end{rem}
\section{Reversibility and positive semi-definiteness of elliptical slice sampling}
\label{sec:rev_ess}	 

With the representation of the transition mechanism of the shrinkage procedure from Algorithm~\ref{alg:shrink} in terms of the transition kernel $Q_S$ we are able to state the transition kernel, say $H$, of elliptical slice sampling that corresponds to the transition mechanism of Algorithm~\ref{alg:ess_ala_ana}. For all $x_{\text{in}}\in \mathcal{H}$ and $A\in\mathcal{B}(\mathcal{H})$ it is given as
\begin{equation}
\label{eq: trans_ESS}
H(x_{\text{in}},A)
= \frac{1}{\rho(x_{\text{in}})}
\int_0^{\rho(x_{\text{in}})} \int_{\mathcal{H}} Q_{S(x_{\text{in}},w,t)}(0,p^{-1}_{x_{\text{in}},w}(\mathcal{H}(t)\cap A))\,  \mu_0({\rm d}w) {\rm d}t,
\end{equation}
where $S(x_{\text{in}},w,t):= p^{-1}_{x_{\text{in}},w}(\mathcal{H}(t))$ for arbitrary $w\in\mathcal{H}$ and $t\in (0,\infty)$.
We also consider the linear operator that is induced by $H$. For this define the Hilbert space
	\[
	L_\mu^2 := \{ f: \mathcal{H} \to \mathbb{R} \colon \int_{\mathcal{H}} f(x)^2 \mu(\dint x) < \infty\}
	\] 
	equipped with the inner product
	$\langle f,g\rangle_{\mu} := \int_{\mathcal{H}} f(x) g(x) \mu(\dint x)$ for $f,g \in L_\mu^2$.	
	Then, $H$ can be extended to a linear operator $\mathrm{H} \colon L_\mu^2 \to L_\mu^2$ given by	
	\begin{equation}\label{eq: ellip_operator}
	f \mapsto \left[ x \mapsto \int_{\mathcal{H}}f(y) H(x, \dint y)\right], \quad f\in L_\mu^2.
	\end{equation}
Here we verify that the reversibility and the positive semi-definiteness of the shrinkage procedure w.r.t.~the uniform distribution on $S(x_{\text{in}},w,t)$ carries over to the reversibility of $H$ w.r.t. $\mu$ and the positive semi-definiteness of $\mathrm{H}$.
\begin{thm}\label{thm: ellip_reversible}
	Let $\mathcal{H}$ be a Hilbert space and $C: \mathcal{H} \to \mathcal{H}$ be a non-singular covariance operator defining $\mu_0 = \mathcal{N}(0,C)$.
		Moreover, let $\rho \colon \mathcal{H} \to (0,\infty)$ be lower-semicontinuous such that $Z:= \int_{\mathcal{H}}\rho(x)\mu_0(\dint x)< \infty$. Then, the transition kernel $H$ of the elliptical slice sampler given in \eqref{eq: def shrinkage operator} is reversible w.r.t. 
		\[
		\mu(\dint x)= \frac{1}{Z}\rho(x)\mu_0(\dint x).
		\]
		Moreover, the linear operator $\mathrm{H} \colon L_\mu^2\to L_\mu^2$ induced by $H$, introduced in \eqref{eq: ellip_operator}, is positive semi-definite.
\end{thm}
We provide some auxiliary tools before proving the main theorem.
\begin{lem}
	\label{lem: same_distr_same_expec}
	Let $X$ and $Y$ be independent random variables mapping to the Hilbert space $\mathcal{H}$, each distributed according to $\mu_0=\mathcal{N}(0,C)$, with $C\colon \mathcal{H} \to \mathcal{H}$ being a non-singular covariance operator. For any $\theta\in[0,2\pi)$ let 
	$T^{(\theta)} \colon \mathcal{H}\times\mathcal{H} \to \mathcal{H}\times\mathcal{H} $ be given by
	\[
	T^{(\theta)} (x,y) := (x \cos \theta + y \sin \theta, x\sin \theta - y \cos \theta).
	\]
	Then
	\begin{align}
	\label{eq: same_distr_same_expec}
	\mathbb{E}(f(\theta,X,Y)) = \mathbb{E}(f(\theta,T^{(\theta)}(X,Y))), \quad \theta\in [0,2\pi),
	\end{align}
	for any $f\colon [0,2\pi) \times \mathcal{H}^2 \to \mathbb{R}$ for which one of the expectations exists, i.e., the product measure of two identically centered normal distributions is invariant under $T^{(\theta)}$.
\end{lem}
\begin{proof}
	By the fact that $X,Y\sim \mu_0=\mathcal{N}(0,C)$ are independent, we have that
	the random vector $\left(\begin{matrix} X\\ Y \end{matrix}\right)$ on $\mathcal{H} \times \mathcal{H}$ is distributed according to 
	$
	N\left(\left(\begin{matrix}
	0 \\ 0
	\end{matrix}\right),\left(
	\begin{matrix}
	C & 0\\
	0 & C
	\end{matrix}
	\right)
	\right).	
	$
	Note that 
	$
	T^{(\theta)} (x,y)^t = \left(
	\begin{matrix}
	\cos \theta I & \sin \theta I\\
	\sin \theta I & -\cos \theta I
	\end{matrix}
	\right) 
	\left(
	\begin{matrix}
	x\\y
	\end{matrix}
	\right),
	$
	where $I\colon \mathcal{H} \to 
	\mathcal{H}$ denotes the identity operator.
	Thus, by the linear transformation theorem for Gaussian measures, see e.g. \cite[Proposition~1.2.3]{da2002second}, we obtain that the vector
	$T^{(\theta)}(X,Y)^t$ is distributed according to
	\begin{align*}
	&	N\left(\left(
	\begin{matrix}
	\cos \theta I & \sin \theta I\\
	\sin \theta I & -\cos \theta I
	\end{matrix}
	\right) 
	\left(\begin{matrix}
	0 \\ 0
	\end{matrix}\right)
	,\left(
	\begin{matrix}
	\cos \theta I & \sin \theta I\\
	\sin \theta I & -\cos \theta I
	\end{matrix}
	\right) \left(\begin{matrix}
	C & 0\\
	0 & C
	\end{matrix}\right) \left(
	\begin{matrix}
	\cos \theta I & \sin \theta I\\
	\sin \theta I & -\cos \theta I
	\end{matrix}
	\right) \right)
	= 	N\left(\left(\begin{matrix}
	0 \\ 0
	\end{matrix}\right),\left(
	\begin{matrix}
	C & 0\\
	0 & C
	\end{matrix}
	\right)\right).
	\end{align*}
	Hence, the distributions of $(X,Y)$ and $T^{(\theta)} (X,Y)$ coincide, such that \eqref{eq: same_distr_same_expec} holds.
\end{proof}
The following useful representation of the inner product allows us to leverage the results for the shrinkage procedure to the elliptical slice sampler.
	\begin{lem}\label{lem: inner_prod_ellip}
		For function $f: \mathcal{H} \to \mathbb{R}$ define 
		\[
		f_{x,y,t}: [0, 2\pi) \to \mathbb{R}, \quad \theta \mapsto \mathbf{1}_{S(x,y,t)}(\theta)(f \circ p_{x,y})(\theta), \quad \quad x,y \in \mathcal{H}, t \in (0, \infty).
		\]
		Under the assumptions of Theorem \ref{thm: ellip_reversible} we have for all $f \in L_\mu^2$ that $f_{x,y,t} \in L_{S(x,y,t)}^2$ for almost all $(x,y,t) \in \{(x,y,t) \in\mathcal{H}^2 \times (0, \infty) \colon x \in \mathcal{H}(t)\}$ w.r.t. $\mu_0\otimes\mu_0\otimes \lambda$.
		Moreover, for all $f,g \in L_\mu^2$ holds
		\begin{equation}\label{eq: inner_prod_ellip}
		\langle \mathrm{H}f,g\rangle_{\mu} =
		\frac{1}{Z} \int_{0}^{\infty} \int_{\mathcal{H}} \int_{\mathcal{H}(t)} \langle \mathrm{Q}_{S(x,y,t)}f_{x,y,t},g_{x,y,t}\rangle_{S(x,y,t)}\  \mu_0(\dint x) \mu_0(\dint y) \dint t. 
		\end{equation}
\end{lem}
\begin{proof}
	We start with some preliminary considerations that are repeatedly used within the proof.
		For arbitrary $\theta \in [0, 2\pi)$ let 
		$g_\theta(\alpha) := (\theta - \alpha)\mod 2\pi$ for $\alpha\in [0,2\pi)$ and note that,
		by using angle sum identities of trigonometric functions, we have
		\[
		p_{T^{(\theta)}(x,y)}(\alpha)
		=
		p_{x,y}(g_\theta(\alpha)),
		\qquad
		\forall \alpha \in [0, 2\pi).
		\]
		By exploiting the previous equality we have for all $ x,y \in \mathcal{H}$ and $t \in (0, \infty)$ that
		\[
		\alpha\in S(T^{(\theta)}(x,y),t) \quad \Longleftrightarrow \quad g_\theta(\alpha) \in S(x,y,t).
		\]
		Thus, $S(T^{(\theta)}(x,y),t) = g_\theta^{-1}(S(x,y,t))$. 
		In particular, we have $\lambda(S(T^{(\theta)}(x,y),t)) = \lambda(S(x,y,t))$, and $\mathbf{1}_{S(T^{(\theta)}(x,y),t)}(\theta) = \mathbf{1}_{S(x,y,t)}(0)=\mathbf{1}_{\mathcal{H}(t)}(x)$.
		
		For $f\in L_\mu^2$ and $X,Y$ being independent $\mu_0$-distributed random variables, by including `active ones', we obtain
		\begin{align*}
		\int_{\mathcal{H}} f(x)^2 \mu(\dint x) & = \frac{1}{Z}
		\int_0^{2\pi} \int_{0}^{\infty} \mathbb{E}[h^{(1)}_t(\theta,X,Y)] \dint t \dint \theta ,
		\end{align*} 
		with 
		$h_t^{(1)}(\theta,x,y) = \mathbf{1}_{\mathcal{H}(t)}(x) \mathbf{1}_{S(x,y,t)}(\theta)f(x)^2\lambda(S(x,y,t))^{-1}$, where we also used that $\mathbf{1}_{(0, \rho(x))}(t) = \mathbf{1}_{\mathcal{H}(t)}(x)$ and $\lambda(S(x,y,t))>0$. The latter holds because $S(x,y,t)$ is open on the circle and non-empty (at least for those $t$ occurring in the expression above) due to the lower semi-continuity of $\rho$.
		Then, Lemma \ref{lem: same_distr_same_expec} and the previous considerations yield
		\begin{align*}
		&	\int_{\mathcal{H}} f(x)^2 \mu(\dint x)  = \frac{1}{Z}
		\int_0^{2\pi} \int_{0}^{\infty} \mathbb{E}[h^{(1)}_t(\theta,T^{(\theta)}(X,Y))] \dint t \dint \theta \\
		=&\int_{0}^{2\pi}\int_{0}^{\infty}\int_{\mathcal{H}}\int_{\mathcal{H}} \frac{\mathbf{1}_{\mathcal{H}(t)}\left(p_{x,y}(\theta)\right) \mathbf{1}_{S(T^{(\theta)}(x,y),t)}(\theta)f\left(p_{x,y}(\theta)\right)^2}{\lambda(S(T^{(\theta)}(x,y),t))Z}\ \mu_0(\dint x)\, \mu_0(\dint y)\, \dint t \, \dint \theta\\
		=&\int_{0}^{2\pi}\int_{0}^{\infty}\int_{\mathcal{H}}\int_{\mathcal{H}} \frac{\mathbf{1}_{\mathcal{H}(t)}(x) f_{x,y,t}(\theta)^2}{\lambda(S(x,y,t))Z}\ \mu_0(\dint x)\, \mu_0(\dint y)\, \dint t \, \dint \theta\\
		=&\frac{1}{Z}\int_{0}^{\infty}\int_{\mathcal{H}}\int_{\mathcal{H}}\mathbf{1}_{\mathcal{H}(t)}(x) \int_{S(x,y,t)} f_{x,y,t}(\theta)^2\ \mathcal{U}_{S(x,y,t)}(\dint \theta) \mu_0(\dint x)\, \mu_0(\dint y)\, \dint t.
		\end{align*}
		This equality implies $f_{x,y,t} \in L_{S(x,y,t)}^2$ for $\mu_0\otimes \mu_0 \otimes \lambda$-almost all $(x,y,t) \in \{(x,y,t) \in\mathcal{H}^2 \times (0, \infty) \colon x \in \mathcal{H}(t)\}$.
		For  proving \eqref{eq: inner_prod_ellip} let $g \in L_\mu^2$. Similar as above, by the definition of $\mathrm{H}$ with
		\[
		\mathrm{Q}_{S(x,y,t)}(f\circ p_{x,y})(0)=\int_{S(x,y,t)} f\left(p_{x,y}(\alpha)\right)\, Q_{S(x,y,t)}(0, \dint \alpha),
		\]
		we obtain
		\begin{align*}
		&\langle \mathrm{H}f,g\rangle_\mu
		=
		\frac{1}{Z} \int_{\mathcal{H}} \mathrm{H}f(x) g(x) \rho(x)\ \mu_0(\dint x)\\
		&
		= \frac{1}{Z} \int_{0}^{\infty} \int_{\mathcal{H}} \int_{\mathcal{H}} g(x) \mathbf{1}_{\mathcal{H}(t)}(x) \mathrm{Q}_{S(x,y,t)}(f\circ p_{x,y})(0)\ \mu_0(\dint x)\mu_0(\dint y) \dint t\\
		& =  \frac{1}{Z} 	\int_{0}^{\infty} \int_{0}^{2\pi}\mathbb{E}\left[h_t^{(2)}(\theta, X, Y)\right] \  \dint \theta\,\dint t,
		\end{align*}
		where $X,Y$ are independent $\mu_0$-distributed random variables and
		\[
		h_t^{(2)}(\theta,x,y)
		=
		\mathbf{1}_{S(x,y,t)}(\theta) g(x) \mathbf{1}_{\mathcal{H}(t)}(x) 
		\mathrm{Q}_{S(x,y,t)}(f\circ p_{x,y})(0)\,
		\lambda(S(x,y,t))^{-1}.
		\]
		Then, by Lemma~\ref{lem: same_distr_same_expec} we have
		\begin{align*}
		\langle \mathrm{H}f,g\rangle_\mu
		&=  \frac{1}{Z} 
		\int_0^\infty
		\int_0^{2\pi}
		\mathbb{E}\left[h^{(2)}_{t}(\theta,T^{(\theta)}(X, Y))\right]
		\dint \theta
		\dint t\\
		&= \frac{1}{Z} 
		\int_0^\infty
		\int_{\mathcal{H}}
		\int_{\mathcal{H}}
		\int_0^{2\pi}
		h^{(2)}_{t}(\theta,T^{(\theta)}(x, y))\ 
		\dint \theta\,
		\mu_0(\dint x)\,
		\mu_0(\dint y)\,
		\dint t.
		\end{align*}
		Taking into account the considerations at the beginning of the proof, Lemma \ref{lem: pushforward_invariance} implies
		\begin{align*}
		& \mathrm{Q}_{S(T^{(\theta)}(x,y),t)}(f\circ p_{T^{(\theta)}(x,y)})(0)  
		=	\int_{0}^{2\pi} \mathbf{1}_{S(T^{(\theta)}(x,y), t)}(\alpha) f(p_{T^{(\theta)}(x,y)}(\alpha)) Q_{S(T^{(\theta)}(x,y),t)}(0, \dint \alpha)\\
		&= \int_{0}^{2\pi} \mathbf{1}_{S\left(x,y, t\right)}\left(g_\theta(\alpha)\right) f\left(p_{x,y}(g_\theta(\alpha))\right) Q_{g_\theta^{-1}(S(x,y,t))}(g_\theta^{-1}(\theta), \dint \alpha)\\
		&= \int_{0}^{2\pi} \mathbf{1}_{S\left(x,y, t\right)}\left(\alpha\right) f\left(p_{x,y}(\alpha)\right) Q_{S(x,y,t)}(\theta, \dint \alpha)\\
		&= \int_{S(x,y,t)} f_{x,y,t}(\alpha) Q_{S(x,y,t)}(\theta, \dint \alpha) = \mathrm{Q}_{S(x,y,t)}f_{x,y,t}(\theta),
		\end{align*}
		and further
		\begin{align*}
		h_t^{(2)}(\theta,T^{(\theta)}(x,y))
		&=
		\frac
		{
			\mathbf{1}_{S(T^{(\theta)}(x,y),t)}(\theta) \cdot g\left(p_{x,y}(\theta)\right) \cdot \mathbf{1}_{\mathcal{H}(t)}\left(p_{x,y}(\theta)\right) \cdot \mathrm{Q}_{S(x,y,t)}f_{x,y,t}(\theta)
		}
		{\lambda(S(T^{(\theta)}(x,y),t))}\\
		&=
		\mathbf{1}_{\mathcal{H}(t)}(x) \cdot g_{x,y,t}(\theta) \cdot \mathrm{Q}_{S(x,y,t)}f_{x,y,t}(\theta)
		\cdot \lambda(S(x,y,t))^{-1}.
		\end{align*}
		Hence
		\begin{align*}
		\int_{0}^{2\pi} h^{(2)}_{t}(\theta,T^{(\theta)}(x, y))\ \dint \theta\,
		&= 	\mathbf{1}_{\mathcal{H}(t)}(x) \int_{S(x,y,t)} g_{x,y,t}(\theta) \ \mathrm{Q}_{S(x,y,t)}f_{x,y,t}(\theta) \ \mathcal{U}_{S(x,y,t)}(\dint \theta)\\
		&= 	\mathbf{1}_{\mathcal{H}(t)}(x) \langle \mathrm{Q}_{S(x,y,t)}f_{x,y,t} ,g_{x,y,t} \rangle_{S(x,y,t)}.
		\end{align*}
		Altogether we obtain
		\begin{align*}
		\langle \mathrm{H}f, g\rangle_{\mu} 
		& = \frac{1}{Z}	\int_0^\infty \int_{\mathcal{H}} \int_{\mathcal{H}(t)} \langle \mathrm{Q}_{S(x,y,t)}f_{x,y,t} ,g_{x,y,t} \rangle_{S(x,y,t)}\  \mu_0(\dint x)\, \mu_0(\dint y)\, \dint t. 
		\qedhere
		\end{align*}
\end{proof}
By combining the previous results we prove the main theorem.
	\begin{proof}[Proof of Theorem \ref{thm: ellip_reversible}]
		Let $A, B \in \mathcal{B}(\mathcal{H})$.
		Since $\rho$ is lower-semicontinuous, $S(x,y,t)$ is open on the circle and non-empty for all $(x,y,t) \in \{(x,y,t) \in \mathcal{H}^2 \times (0,\infty) \colon x \in \mathcal{H}(t)\}$.
		By Theorem~\ref{thm: shrink_reversible} this yields
		\begin{align*}
		&\langle \mathrm{Q}_{S(x,y,t)}\mathbf{1}_{p_{x,y}^{-1}(A \cap \mathcal{H}(t))},\mathbf{1}_{p_{x,y}^{-1}(B \cap \mathcal{H}(t))}\rangle_{S(x,y,t)}\\
		=& \int_{p_{x,y}^{-1}(B \cap \mathcal{H}(t))} Q_{S(x,y,t)}(\theta, p_{x,y}^{-1}(A \cap \mathcal{H}(t)))\, \mathcal{U}_{S(x,y,t)}(\dint \theta)\\
		=&\int_{p_{x,y}^{-1}(A \cap \mathcal{H}(t))} Q_{S(x,y,t)}(\theta, p_{x,y}^{-1}(B \cap \mathcal{H}(t)))\, \mathcal{U}_{S(x,y,t)}(\dint \theta)\\
		=&\langle \mathrm{Q}_{S(x,y,t)} \mathbf{1}_{p_{x,y}^{-1}(B \cap \mathcal{H}(t))},\mathbf{1}_{p_{x,y}^{-1}(A \cap \mathcal{H}(t))}\rangle_{S(x,y,t)}.
		\end{align*}
		By $(\mathbf{1}_{A})_{x,y,t} = \mathbf{1}_{p_{x,y}^{-1}(A \cap \mathcal{H}(t))}$, we get with	 \eqref{eq: inner_prod_ellip} that
		\begin{align*}
		&\int_{B} H(x, A) \mu(\dint x) 
		= \langle \mathrm{H}\mathbf{1}_{A},\mathbf{1}_{B}\rangle_\mu \\
		&\qquad= \frac{1}{Z} \int_{0}^{\infty} \int_{\mathcal{H}} \int_{\mathcal{H}(t)} \langle \mathrm{Q}_{S(x,y,t)} \mathbf{1}_{p_{x,y}^{-1}(A \cap \mathcal{H}(t))},\mathbf{1}_{p_{x,y}^{-1}(B \cap \mathcal{H}(t))}\rangle_{S(x,y,t)}\ \mu_0(\dint x)\, \mu_0(\dint y) \, \dint t\\
		&\qquad= \frac{1}{Z} \int_{0}^{\infty} \int_{\mathcal{H}} \int_{\mathcal{H}(t)} \langle 
		\mathrm{Q}_{S(x,y,t)}\mathbf{1}_{p_{x,y}^{-1}(B \cap \mathcal{H}(t))},\mathbf{1}_{p_{x,y}^{-1}(A \cap \mathcal{H}(t))}\rangle_{S(x,y,t)}\ \mu_0(\dint x)\, \mu_0(\dint y) \, \dint t\\
		&\qquad = \langle \mathrm{H}\mathbf{1}_{B},\mathbf{1}_{A}\rangle_\mu = \int_{A} H(x, B) \mu(\dint x),
		\end{align*}
		which verifies the desired reversibility w.r.t. $\mu$. We now turn to positive semi-definiteness.
		Combining Lemma \ref{lem: inner_prod_ellip} and Lemma \ref{lem: positive semidefiniteness} yields for all $f \in L_\mu^2$ that
		\begin{align*}
		\langle \mathrm{H}f,f\rangle_\mu 
		& = \frac{1}{Z}\int_{0}^{\infty} \int_{\mathcal{H}} \int_{\mathcal{H}(t)} \langle \mathrm{Q}_{S(x,y,t)}f_{x,y,t},f_{x,y,t}\rangle_{S(x,y,t)}\  \mu_0(\dint x) \mu_0(\dint y) \dint t
		\geq 0.
		\qedhere
		\end{align*}
	\end{proof}
\section{Summary and outlook}

Let us summarize our main findings. We provide a proof of reversibility and positive semi-definiteness of ESS, where the underlying state space of the corresponding Markov chain can be an infinite-dimensional Hilbert space. On the way to that we point to a (weak) qualitative regularity condition of the likelihood function ($\rho$ is assumed to be lower semicontinuous) that guarantees that the appearing while loop terminates and therefore leads to a well-defined transition kernel. Moreover, with \eqref{eq: trans_ESS} we developed a representation of the transition kernel of ESS. Our approach illuminates the hybrid slice sampling structure, cf. \cite{latuszynski2014convergence}, in terms of the positive semi-definiteness, see Lemma~\ref{lem: positive semidefiniteness}, and the reversibility of the shrinkage procedure, see Theorem~\ref{thm: shrink_reversible}, w.r.t. the uniform distribution on a subset of the angle space $[0,2\pi)$.

The formerly developed representations and tools might path the way for an analysis of the spectral gap of ESS regarding dimension independent behavior. A strictly positive spectral gap is a desirable property of a Markov chain w.r.t.~mixing properties as well as the theoretical assessment of the mean squared error of Markov chain Monte Carlo for the approximations of expectations according to $\mu$, for details see for example \cite{rudolf2012explicit}. Coupling constructions as have been derived for simple slice sampling in \cite{natarovskii2021quantitative} might be a promising approach for addressing the verification of the existence of such a positive spectral gap on $\mathcal{H}$. Moreover, it also seems advantageous to further explore the structural similarity between Metropolis-Hastings and slice sampling approaches. In particular, approximate (elliptical) slice sampling that relies on evaluations of proxys of the likelihood function are interesting. Here stability investigations as e.g. delivered in \cite{habeck2020stability,sprungk2020local} might be used to obtain perturbation theoretical results for ESS as presented in \cite{rudolf2018perturbation,medina2020perturbation} for approximate Metropolis Hastings. 
Eventually, the theoretical investigation of ESS might be useful to verify the reversibility property and the positive semi-definiteness property also for other slice sampling schemes that rely on the shrinkage procedure.

\begin{appendix}
\section{Technical proofs}

\subsection{Proof of Lemma~\ref{lem: theta_conditioned_on_interval}}
\label{app: 1}

\begin{proof}
	By induction over $n\in\mathbb{N}$ we prove 
	\begin{equation}
	\label{eq: to_show}
	\mathbb{E}[\mathbf{1}_A(\Theta)\mid Z_1,\dots,Z_n] = 
	\frac{\mathbb{P}_{\Theta}(A\cap I(\Gamma_n^{\min},\Gamma_n^{\max}))}{\mathbb{P}_\Theta(I(\Gamma_n^{\min},\Gamma_n^{\max}))},
	\end{equation}
	from which the statement follows readily.
	We start with the base case, i.e., consider $n=1$. Note that $Z_1=(\Gamma_1,\Gamma_1^{\min},\Gamma_1^{\max})$ is independent of $\Theta$ and that $I(\Gamma_1^{\min},\Gamma_1^{\max}) = I(\Gamma_1,\Gamma_1) = [0,2\pi)$. Using those properties yields
	\begin{align*}
	\mathbb{E}(\mathbf{1}_A(\Theta)\mid Z_1) & = \mathbb{P}_\Theta (A) 
	= \frac{\mathbb{P}_\Theta(A\cap I(\Gamma_1,\Gamma_1))}{\mathbb{P}_\Theta(I(\Gamma_1,\Gamma_1))} 
	= \frac{\mathbb{P}_\Theta(A\cap I(\Gamma^{\min}_1,\Gamma^{\max}_1))}{\mathbb{P}_\Theta(I(\Gamma^{\min}_1,\Gamma^{\max}_1))},
	\end{align*}
	which verifies \eqref{eq: to_show} for $n=1$.
	
	Assume that \eqref{eq: to_show} is true for $n$, we are going to prove it for $n+1$. 
	Observe that, as $Z_n \in \Lambda_\Theta$ almost surely, we have $\Theta \in I(\Gamma_n^{\min},\Gamma_n^{\max})$, $\Gamma_n \in I(\Gamma_n^{\min},\Gamma_n^{\max})$ and the following two implications
	\begin{align}
	\label{al: impl_theta1}
	\Theta \in I(\Gamma_n,\Gamma_n^{\max}) & \; \Longrightarrow \; 
	 \Gamma_n\in J(\Gamma_n^{\min},\Theta) 
	\; \Longrightarrow \;
	\Gamma_{n+1}^{\min} = \Gamma_n, \; \Gamma_{n+1}^{\max} = \Gamma_n^{\max},\\
	\label{al: impl_theta2}
	\Theta \in J(\Gamma_n^{\min},\Gamma_n) & \; \Longrightarrow \;
	\Gamma_n\in I(\Theta,\Gamma_n^{\max}) 
	\; \Longrightarrow \;
	\Gamma_{n+1}^{\min} = \Gamma_n^{\min}, \; \Gamma_{n+1}^{\max} = \Gamma_n.
	\end{align}
	Moreover, by the induction assumption, the fact that $\Gamma_n \in I(\Gamma_n^{\min},\Gamma_n^{\max})$ almost surely and disintegration, {see \cite[Chapter~6, Theorem~6.4 and (6)]{Ka02}}, we have 
	\begin{align}
	\label{al: gamma_gamma_max}
	\mathbb{E}[\mathbf{1}_{I(\Gamma_n,\Gamma_n^{\max})}(\Theta)\mid Z_1,\dots,Z_n]  
	& = \frac{\mathbb{P}_\Theta(I(\Gamma_n,\Gamma_n^{\max}))}{\mathbb{P}_\Theta(I(\Gamma_n^{\min},\Gamma_n^{\max}))},\\
	\label{al: gamma_min_gamma}
	\mathbb{E}[\mathbf{1}_{J(\Gamma^{\min}_n,\Gamma_n)}(\Theta)\mid Z_1,\dots,Z_n]  
	& = \frac{\mathbb{P}_\Theta(J(\Gamma^{\min}_n,\Gamma_n))}{\mathbb{P}_\Theta(I(\Gamma_n^{\min},\Gamma_n^{\max}))} =  \mathbf{1}_{\{\Gamma^{\min}_n\neq\Gamma_n\}}\frac{\mathbb{P}_\Theta(I(\Gamma^{\min}_n,\Gamma_n))}{\mathbb{P}_\Theta(I(\Gamma_n^{\min},\Gamma_n^{\max}))}.
	\end{align}
	For arbitrary $A\in \mathcal{B}([0,2\pi))$, $C_i \in \mathcal{B}([0,2\pi)^3)$ with $i=1,\dots,n+1$  we verify
	\begin{align}
	& \mathbb{E}\Big[\mathbf{1}_A(\Theta) \prod_{i=1}^{n+1} \mathbf{1}_{C_i}(Z_i)\Big]
	= \mathbb{E}\Big[ 
	\prod_{i=1}^{n+1} \mathbf{1}_{C_i}(Z_i) \frac{\mathbb{P}_\Theta(A\cap I(\Gamma^{\min}_{n+1},\Gamma_{n+1}^{\max}))}{\mathbb{P}_\Theta(I(\Gamma^{\min}_{n+1},\Gamma_{n+1}^{\max}))} \Big].
	\label{al: to_show}
	\end{align}
	Hence by the definition of the conditional distribution/expectation and the fact that Cartesian product sets of the above form generate the $\sigma$-algebra $\mathcal{B}([0,2\pi)^{3(n+1)})$ we obtain \eqref{eq: to_show}. For proving \eqref{al: to_show} we observe that
	\begin{align*}
	& \mathbb{E}\Big[\mathbf{1}_A(\Theta)\prod_{i=1}^{n+1} \mathbf{1}_{C_i}(Z_i)\Big]
	=  \mathbb{E}\Big [ \mathbb{E} \Big[ \mathbf{1}_{A}(\Theta) \prod_{i=1}^{n+1} \mathbf{1}_{C_i}(Z_i)  \mid Z_1,\dots,Z_n, \Theta \Big] \Big]\\
	=\; & \mathbb{E}\Big [  \mathbf{1}_{A}(\Theta) \prod_{i=1}^{n} \mathbf{1}_{C_i}(Z_i)\, \mathbb{P}\left( Z_{n+1}\in C_{n+1}   \mid Z_1,\dots,Z_n, \Theta \right) \Big]\\
	=\; & \mathbb{E}\Big [  \mathbf{1}_{A}(\Theta) \prod_{i=1}^{n} \mathbf{1}_{C_i}(Z_i)\,\mathbb{P}\left( Z_{n+1}\in C_{n+1}   \mid Z_n, \Theta \right) \Big]
	=\;  \mathbb{E}\Big [  \mathbf{1}_{A}(\Theta) \prod_{i=1}^{n} \mathbf{1}_{C_i}(Z_i)\, R_{\Theta}(Z_n,C_{n+1}) \Big]. 
	\end{align*}
	By Lemma~\ref{lem: repres_R_theta} we conclude from the previous calculation that
	\begin{align*}
	& \mathbb{E}\Big[\mathbf{1}_A(\Theta)\prod_{i=1}^{n+1} \mathbf{1}_{C_i}(Z_i)\Big]
	=\; \mathbb{E}\Big[\prod_{i=1}^{n}\mathbf{1}_{C_i}(Z_i)\, \mathbf{1}_{A\cap I(\Gamma_n,\Gamma_n^{\max})}(\Theta) 
	\,\mathcal{U}_{I(\Gamma_n,\Gamma_n^{\max})}\otimes\delta_{(\Gamma_n,\Gamma_n^{\max})}(C_{n+1}) \Big]\\
	& \qquad \qquad \qquad \qquad \qquad + \mathbb{E}\Big[\prod_{i=1}^{n}\mathbf{1}_{C_i}(Z_i)\, \mathbf{1}_{A\cap J(\Gamma_n^{\min},\Gamma_n)}(\Theta) \,
	\mathcal{U}_{I(\Gamma_n^{\min},\Gamma_n)}\otimes\delta_{(\Gamma_n^{\min},\Gamma_n)}(C_{n+1}) \Big]\\
	=\;& \mathbb{E}\Big[\prod_{i=1}^{n}\mathbf{1}_{C_i}(Z_i)\,
	\mathcal{U}_{I(\Gamma_n,\Gamma_n^{\max})}\otimes\delta_{(\Gamma_n,\Gamma_n^{\max})}(C_{n+1})\, \mathbb{E}[\mathbf{1}_{A\cap I(\Gamma_n,\Gamma_n^{\max})}(\Theta)  \mid Z_1,\dots,Z_n] \Big]\\
	& + \mathbb{E}\Big[\prod_{i=1}^{n}\mathbf{1}_{C_i}(Z_i)\,
	\mathcal{U}_{I(\Gamma_n^{\min},\Gamma_n)}\otimes\delta_{(\Gamma_n^{\min},\Gamma_n)}(C_{n+1}) \, \mathbb{E}[\mathbf{1}_{A\cap J(\Gamma_n^{\min},\Gamma_n)}(\Theta) \mid Z_1,\dots,Z_n]\Big].
	\end{align*}
	For abbreviating the notation define
	\begin{align*}
	H_{\max}(Z_1,\dots,Z_n) & := \prod_{i=1}^{n}\mathbf{1}_{C_i}(Z_i)\,
	\mathcal{U}_{I(\Gamma_n,\Gamma_n^{\max})}\otimes\delta_{(\Gamma_n,\Gamma_n^{\max})}(C_{n+1}),\\
	H_{\min}(Z_1,\dots,Z_n) & :=
	\prod_{i=1}^{n}\mathbf{1}_{C_i}(Z_i)\,
	\mathcal{U}_{I(\Gamma_n^{\min},\Gamma_n)}\otimes\delta_{(\Gamma_n^{\min},\Gamma_n)}(C_{n+1}).
	\end{align*} 
	Observe that $\mathbb{P}_\Theta(A\cap J(\Gamma^{\min}_n,\Gamma_n)) = \mathbf{1}_{\{\Gamma^{\min}_n\neq\Gamma_n\}}\,\mathbb{P}_\Theta(A\cap I(\Gamma^{\min}_n,\Gamma_n))$. Therefore
	using the induction assumption and the fact that $\Gamma_n\in I(\Gamma_n^{\min},\Gamma_n^{\max})$ almost surely implies $J(\Gamma_n^{\min},\Gamma_n) \subset I(\Gamma_n^{\min},\Gamma_n^{\max})$
	and $I(\Gamma_n,\Gamma_n^{\max}) \subset I(\Gamma_n^{\min},\Gamma_n^{\max})$ almost surely we have
	\begin{align*}
	& \mathbb{E}\Big[\mathbf{1}_A(\Theta)\prod_{i=1}^{n+1} \mathbf{1}_{C_i}(Z_i)\Big]\\
	=\;& \mathbb{E}\Big[H_{\max}(Z_1,\dots,Z_n)\, \frac{\mathbb{P}_\Theta(A\cap I(\Gamma_n,\Gamma_n^{\max}))}{\mathbb{P}_\Theta(I(\Gamma_n^{\min},\Gamma_n^{\max}))} \Big]\\
	& \quad  + \mathbb{E}\Big[H_{\min}(Z_1,\dots,Z_n) {\mathbf{1}_{\{\Gamma_n^{\min}\neq\Gamma_n\}}}\, \frac{\mathbb{P}_\Theta(A\cap I(\Gamma^{\min}_n,\Gamma_n))}{\mathbb{P}_\Theta(I(\Gamma_n^{\min},\Gamma_n^{\max}))}\Big]\\
	=\;& \mathbb{E}\Big[H_{\max}(Z_1,\dots,Z_n)\,\frac{\mathbb{P}_\Theta(I(\Gamma_n,\Gamma_n^{\max}))}{\mathbb{P}_\Theta(I(\Gamma_n^{\min},\Gamma_n^{\max}))} \frac{\mathbb{P}_\Theta(A\cap I(\Gamma_n,\Gamma_n^{\max}))}{\mathbb{P}_\Theta(I(\Gamma_n,\Gamma_n^{\max}))} \Big]\\
	& \quad  + \mathbb{E}\Big[H_{\min}(Z_1,\dots,Z_n) {\mathbf{1}_{\{\Gamma_n^{\min}\neq\Gamma_n\}}}\,\frac{\mathbb{P}_\Theta(I(\Gamma_n^{\min},\Gamma_n))}{\mathbb{P}_\Theta(I(\Gamma_n^{\min},\Gamma_n^{\max}))} \frac{\mathbb{P}_\Theta(A\cap I(\Gamma^{\min}_n,\Gamma_n))}{\mathbb{P}_\Theta(I(\Gamma_n^{\min},\Gamma_n))} \Big].
	\end{align*}
	Using \eqref{al: gamma_gamma_max} as well as \eqref{al: gamma_min_gamma} we get
	\begin{align*}
	& \mathbb{E}\Big[\mathbf{1}_A(\Theta)\prod_{i=1}^{n+1} \mathbf{1}_{C_i}(Z_i)\Big]\\
	=\;& \mathbb{E}\Big[H_{\max}(Z_1,\dots,Z_n)\,\mathbb{E}[\mathbf{1}_{I(\Gamma_n,\Gamma_n^{\max})}(\Theta)\mid Z_1,\dots,Z_n]
	\cdot \frac{\mathbb{P}_\Theta(A\cap I(\Gamma_n,\Gamma_n^{\max}))}{\mathbb{P}_\Theta(I(\Gamma_n,\Gamma_n^{\max}))} \Big]\\
	& + \mathbb{E}\Big[H_{\min}(Z_1,\dots,Z_n) \,\mathbb{E}[\mathbf{1}_{J(\Gamma_n^{\min},\Gamma_n)}(\Theta)\mid Z_1,\dots,Z_n]
	\cdot
	\frac{\mathbb{P}_\Theta(A\cap I(\Gamma^{\min}_n,\Gamma_n))}{\mathbb{P}_\Theta(I(\Gamma_n^{\min},\Gamma_n))}\Big]
	\\
	=\;& \mathbb{E}\Big[\mathbb{E}\Big[H_{\max}(Z_1,\dots,Z_n)\,\mathbf{1}_{I(\Gamma_n,\Gamma_n^{\max})}(\Theta)\cdot \frac{\mathbb{P}_\Theta(A\cap I(\Gamma_n,\Gamma_n^{\max}))}{\mathbb{P}_\Theta(I(\Gamma_n,\Gamma_n^{\max}))}\mid Z_1,\dots,Z_n\Big] \Big]\\
	& + \mathbb{E}\Big[\mathbb{E}\Big[H_{\min}(Z_1,\dots,Z_n) \,\mathbf{1}_{J(\Gamma_n^{\min},\Gamma_n)}(\Theta)
	\cdot
	\frac{\mathbb{P}_\Theta(A\cap I(\Gamma^{\min}_n,\Gamma_n))}{\mathbb{P}_\Theta(I(\Gamma_n^{\min},\Gamma_n))}\mid Z_1,\dots,Z_n\Big]\Big].
	\end{align*}
	Denoting 
	\[
	\mathcal{T}_{I(\Gamma^{\min}_{n+1},\Gamma_{n+1}^{\max})}(A) := 
	\frac{\mathbb{P}_\Theta(A\cap I(\Gamma^{\min}_{n+1},\Gamma_{n+1}^{\max}))}{\mathbb{P}_\Theta(I(\Gamma^{\min}_{n+1},\Gamma_{n+1}^{\max}))}
	\]
	and exploiting \eqref{al: impl_theta1} as well as \eqref{al: impl_theta2} gives
	\begin{align*}
	& \mathbb{E}\Big[\mathbf{1}_A(\Theta)\prod_{i=1}^{n+1} \mathbf{1}_{C_i}(Z_i)\Big]\\
	=\;& \mathbb{E}\Big[ \prod_{i=1}^{n}\mathbf{1}_{C_i}(Z_i)\,
	\mathcal{U}_{I(\Gamma^{\min}_{n+1},\Gamma_{n+1}^{\max})}\otimes\delta_{(\Gamma^{\min}_{n+1},\Gamma_{n+1}^{\max})}(C_{n+1})\,\mathbf{1}_{I(\Gamma_n,\Gamma_n^{\max})}(\Theta) \mathcal{T}_{I(\Gamma^{\min}_{n+1},\Gamma_{n+1}^{\max})}(A)\Big]\\
	& + \mathbb{E}\Big[
	\prod_{i=1}^{n}\mathbf{1}_{C_i}(Z_i)\,
	\mathcal{U}_{I(\Gamma_{n+1}^{\min},\Gamma^{\max}_{n+1})}\otimes\delta_{(\Gamma_{n+1}^{\min},\Gamma^{\max}_{n+1})}(C_{n+1}) \,\mathbf{1}_{J(\Gamma_n^{\min},\Gamma_n)}(\Theta)
	\mathcal{T}_{I(\Gamma^{\min}_{n+1},\Gamma_{n+1}^{\max})}(A)\Big].
	\end{align*}
	By the fact that $\Theta,\Gamma_n \in I(\Gamma_n^{\min},\Gamma_n^{\max})$ almost surely, we have
	$\mathbf{1}_{I(\Gamma_n,\Gamma_n^{\max})}(\Theta) +\mathbf{1}_{J(\Gamma_n^{\min},\Gamma_n)}(\Theta)=1$ almost surely, such that
	\begin{align*}
	& \mathbb{E}\Big[\mathbf{1}_A(\Theta)\prod_{i=1}^{n+1} \mathbf{1}_{C_i}(Z_i)\Big]
	=\; \mathbb{E}\Big[ \prod_{i=1}^{n}\mathbf{1}_{C_i}(Z_i)\,
	\mathcal{U}_{I(\Gamma^{\min}_{n+1},\Gamma_{n+1}^{\max})}\otimes\delta_{(\Gamma^{\min}_{n+1},\Gamma_{n+1}^{\max})}(C_{n+1})\,\mathcal{T}_{I(\Gamma^{\min}_{n+1},\Gamma_{n+1}^{\max})}(A)\Big].
	\end{align*}
	By virtue of \eqref{al: gamma_given_interval} we have 
	\begin{align*}
	\mathbb{E}[\mathbf{1}_{C_{n+1}}(Z_{n+1})\mid \Gamma_{n+1}^{\min},\Gamma_{n+1}^{\max}] & =   \mathcal{U}_{I(\Gamma^{\min}_{n+1},\Gamma_{n+1}^{\max})}\otimes\delta_{(\Gamma^{\min}_{n+1},\Gamma_{n+1}^{\max})}(C_{n+1})\\
	&=  \mathbb{E}[\mathbf{1}_{C_{n+1}}(Z_{n+1})\mid Z_1,\dots,Z_n, \Gamma_{n+1}^{\min},\Gamma_{n+1}^{\max}] ,
	\end{align*}
	such that
	\begin{align*}
	 \mathbb{E}\Big[\mathbf{1}_A(\Theta)\prod_{i=1}^{n+1} \mathbf{1}_{C_i}(Z_i)\Big]
	=\;& \mathbb{E}\Big[ \prod_{i=1}^{n}\mathbf{1}_{C_i}(Z_i)\,\mathbb{E}[\mathbf{1}_{C_{n+1}}(Z_{n+1})\mid Z_1,\dots,Z_n, \Gamma_{n+1}^{\min},\Gamma_{n+1}^{\max}]\,\mathcal{T}_{I(\Gamma^{\min}_{n+1},\Gamma_{n+1}^{\max})}(A)\Big]\\
	=\;& \mathbb{E}\Big[\mathbb{E}\Big[ \prod_{i=1}^{n+1}\mathbf{1}_{C_i}(Z_i)\,\mathcal{T}_{I(\Gamma^{\min}_{n+1},\Gamma_{n+1}^{\max})}(A)
	\mid Z_1,\dots,Z_n, \Gamma_{n+1}^{\min},\Gamma_{n+1}^{\max}\Big]\Big]\\
	=\;& \mathbb{E}\Big[ \prod_{i=1}^{n+1}\mathbf{1}_{C_i}(Z_i)\,\frac{\mathbb{P}_\Theta(A\cap I(\Gamma^{\min}_{n+1},\Gamma_{n+1}^{\max}))}{\mathbb{P}_\Theta(I(\Gamma^{\min}_{n+1},\Gamma_{n+1}^{\max}))}\Big].
	\end{align*}
	Finally observing that $\frac{\mathbb{P}_\Theta(A\cap I(\Gamma^{\min}_{n+1},\Gamma_{n+1}^{\max}))}{\mathbb{P}_\Theta(I(\Gamma^{\min}_{n+1},\Gamma_{n+1}^{\max}))}$ is measurable w.r.t. $\sigma(Z_1,\dots,Z_{n+1})$ we have proven \eqref{eq: to_show} and the total statement is verified. 
\end{proof}

\subsection{Proof of Lemma~\ref{lem: pushforward_invariance}}
\label{sec: proof_pushforward_invariance}
	We start with stating notation and proving auxiliary results that are used in the actual proof of Lemma~\ref{lem: pushforward_invariance}.
	For $\alpha,\beta\in [0,2\pi)$ let
	\[ 
	{\inI}(\alpha, \beta):= \begin{cases}	(\alpha, \beta), & \alpha < \beta\\
	[0, \beta) \cup (\alpha, 2 \pi), & \alpha > \beta\\
	[0, 2\pi), & \alpha = \beta.
	\end{cases} 
	\]
	We provide two lemmas, where the first follows by 
	performing a case distinction.
	\begin{lem}  \label{lem: interval_repres_after_gtheta}
		Let $\theta\in [0,2\pi)$ and define $g_\theta \colon [0,2\pi) \to [0,2\pi)$ by $g_\theta (\alpha) = (\theta - \alpha) \mod 2\pi$.
		Then, for any $\alpha,\beta\in [0,2\pi)$ we have 
		\[
		g_\theta(I^
		\circ(\alpha,\beta)) = I^
		\circ(g_\theta(\beta),g_\theta(\alpha)).
		\]
	\end{lem}
	\begin{lem} \label{lem: a_s_R_alpha}
		Let $\alpha,\theta \in (0,2\pi]$ and
		\[
		G_{\alpha} := \{(\gamma,\gamma^{\min},\gamma^{\max})\in [0,2\pi)^3 \colon \alpha, \gamma \in I^
		\circ(\gamma^{\min}, \gamma^{\max}), \alpha\neq\gamma,\alpha \neq \gamma^{\min}, \alpha \neq \gamma^{\max}\}.
		\]
		Define
		$\widetilde{g}_\theta \colon [0,2\pi)^3 \to [0,2\pi)^3$ by $\widetilde{g}_\theta(\gamma_1,\gamma_2,\gamma_3) = (g_\theta(\gamma_1),g_\theta(\gamma_3),g_\theta(\gamma_2))$,
		where $g_\theta$ is the mapping from Lemma~\ref{lem: interval_repres_after_gtheta}.
		Then, for $R_\alpha$ (introduced before Lemma~\ref{lem: repres_R_theta}) we have for any $z\in G_\alpha$ that $R_\alpha(z,G_\alpha)=1$ and 
		\begin{equation}
		\label{eq: pushforward_coincides}
		R_{g_\theta^{-1}(\alpha)}(\widetilde{g}_\theta(z),\widetilde{g}_\theta(C)) = R_\alpha(z,C),
		\end{equation}
		for any $C\in \mathcal{B}([0,2\pi)^3)$.
	\end{lem}
	\begin{proof}
		Observe that for all $\alpha, \beta \in [0,2\pi)$ we have $\mathcal{U}_{I(\alpha, \beta)} = \mathcal{U}_{I^
			\circ(\alpha, \beta)} $, and for all $z= (\gamma, \gamma^{\min}, \gamma^{\max}) \in G_\alpha$ holds
		\begin{equation*}
		\mathbf{1}_{I(\gamma, \gamma^{\max})}(\alpha) = \mathbf{1}_{I^\circ(\gamma, \gamma^{\max})}(\alpha), \qquad 
		\mathbf{1}_{J(\gamma^{\min},\gamma)}(\alpha) = \mathbf{1}_{[0,2\pi)\setminus\{\gamma^{\min}\}}(\gamma)\,\mathbf{1}_{I^\circ(\gamma^{\min},\gamma)}(\alpha).
		\end{equation*}
		Therefore we get by Lemma \ref{lem: repres_R_theta} for all $z= (\gamma, \gamma^{\min}, \gamma^{\max}) \in G_\alpha$ and any $C\in \mathcal{B}([0,2\pi)^3)$ that
		\begin{align}
		\label{eq: repres_R_alpha_on_G_alpha}
		R_\alpha((\gamma,\gamma^{\min},\gamma^{\max}), C) 
		&= \mathbf{1}_{I^\circ(\gamma,\gamma^{\max})}(\alpha) \cdot 
		\mathcal{U}_{I^\circ(\gamma,\gamma^{\max})} \otimes 
		\delta_{(\gamma,\gamma^{\max})}(C) \\
		&\quad+ \mathbf{1}_{[0,2\pi)\setminus\{\gamma^{\min}\}}(\gamma)\,
		\mathbf{1}_{I^\circ(\gamma^{\min},\gamma)}(\alpha) \cdot 
		\mathcal{U}_{I^\circ(\gamma^{\min},\gamma)} \otimes 
		\delta_{(\gamma^{\min},\gamma)}(C). \notag
		\end{align}
		Now fix  $z=(\gamma,\gamma^{\min},\gamma^{\max})\in G_\alpha$ and $C\in \mathcal{B}([0,2\pi)^3)$.
		Since $z\in G_\alpha$ implies $\widetilde{g}_{\theta}(z)\in G_{{g^{-1}_\theta}(\alpha)}$ and $g_\theta = g_\theta^{-1}$, $\widetilde{g}_\theta = \widetilde{g}_\theta^{-1}$, we obtain by \eqref{eq: repres_R_alpha_on_G_alpha}
		\begin{align*}
		&R_{g_\theta^{-1}(\alpha)}(\widetilde{g}_\theta(z),\widetilde{g}_\theta(C))
		%
		%
		= \mathbf{1}_{g_\theta(I^\circ(g_\theta(\gamma),g_\theta(\gamma^{\min})))}(\alpha) 
		\cdot \mathcal{U}_{g_\theta(I^\circ(g_\theta(\gamma),g_\theta(\gamma^{\min})))} \otimes 
		\delta_{(\gamma^{\min},\gamma)}(C) \\
		& \quad +\mathbf{1}_{[0,2\pi)\setminus(g_\theta(\gamma^{\max}))}(g_\theta(\gamma)) \mathbf{1}_{g_\theta(I^
			\circ(g_\theta(\gamma^{\max}),g_\theta(\gamma)))}(\alpha) \cdot \mathcal{U}_{g_\theta(I^
			\circ(g_\theta(\gamma^{\max}),g_\theta(\gamma)))} \otimes 
		\delta_{(\gamma,\gamma^{\max})}(C).
		\end{align*}
		Lemma~\ref{lem: interval_repres_after_gtheta} and again $g_\theta=g_\theta^{-1}$ yield
		\begin{align*}
		R_{g_\theta^{-1}(\alpha)}(\widetilde{g}_\theta(z),\widetilde{g}_\theta(C))
		&= \mathbf{1}_{I^
			\circ(\gamma^{\min},\gamma)}(\alpha) \cdot \mathcal{U}_{I^
			\circ(\gamma^{\min},\gamma)} \otimes 
		\delta_{(\gamma^{\min},\gamma)}(C) \\
		& \quad + \mathbf{1}_{[0,2\pi)\setminus\{\gamma^{\max}\}}(\gamma)\, \mathbf{1}_{I^
			\circ(\gamma,\gamma^{\max})}(\alpha) \cdot \mathcal{U}_{I^
			\circ(\gamma,\gamma^{\max})} \otimes \delta_{(\gamma,\gamma^{\max})}(C).
		\end{align*}
		Observe that $z \in G_\alpha$ implies that $\gamma = \gamma^{\max}$ or $\gamma = \gamma^{\min}$ is only possible if $\gamma=\gamma^{\min} = \gamma^{\max}$.
		Therefore 
		\begin{align*}
		R_{g_\theta^{-1}(\alpha)}(\widetilde{g}_\theta(z),\widetilde{g}_\theta(C))
		&= \mathbf{1}_{[0,2\pi)\setminus\{\gamma^{\min}\}}(\gamma)\,\mathbf{1}_{I^
			\circ(\gamma^{\min},\gamma)}(\alpha) \cdot \mathcal{U}_{I^
			\circ(\gamma^{\min},\gamma)} \otimes 
		\delta_{(\gamma^{\min},\gamma)}(C) \\
		& \quad +  \mathbf{1}_{I^
			\circ(\gamma,\gamma^{\max})}(\alpha) \cdot \mathcal{U}_{I^
			\circ(\gamma,\gamma^{\max})} \otimes 
		\delta_{(\gamma,\gamma^{\max})}(C) = R_\alpha(z,C), 
		\end{align*}
		where the last equality follows by \eqref{eq: repres_R_alpha_on_G_alpha}.
		
		Next we argue for $R_\alpha(z, G_\alpha)= 1$.
		Again from \eqref{eq: repres_R_alpha_on_G_alpha} and the definition of $G_\alpha$ we get
		\begin{align*}
		& R_\alpha((\gamma,\gamma^{\min},\gamma^{\max}), G_\alpha) \\
		& = \mathbf{1}_{I^\circ(\gamma,\gamma^{\max})}(\alpha)  \cdot \mathcal{U}_{I^\circ(\gamma,\gamma^{\max})}
		\left(I^\circ(\gamma,\gamma^{\max}) \setminus \{\alpha\}\right)\\
		&\qquad + \mathbf{1}_{[0,2\pi)\setminus\{\gamma^{\min}\}}(\gamma)\,\mathbf{1}_{I^
			\circ(\gamma^{\min},\gamma)}(\alpha) \cdot \mathcal{U}_{I^
			\circ(\gamma^{\min},\gamma)} \left(I^
		\circ(\gamma^{\min},\gamma) \setminus\{\alpha\}\right)\\
		&= \mathbf{1}_{I^\circ(\gamma,\gamma^{\max})}(\alpha) + \mathbf{1}_{J(\gamma^{\min},\gamma)}(\alpha)=1.
		\qedhere
		\end{align*}
	\end{proof}
	We recall the notion of a tensor product of kernels and probability measures. For $\alpha\in[0,2\pi)$ let $(Z_k)_{k\in\N}$ be the shrinkage Markov chain with transition kernel $R_\alpha$ on $\Lambda_\alpha$ and initial distribution $\mathbb{P}_{Z_1}$ see \eqref{eq: init_distr}. Then, for $k\in\N$ the tensor product measure $\mathbb{P}_{Z_1}\otimes R_\alpha^{\otimes(k-1)}$ on the product space $([0,2\pi)^{3k},\mathcal{B}([0,2\pi)^{3k}))$ is the uniquely determined probability measure that satisfies 
	\begin{align*}
	\mathbb{P}_{Z_1}\otimes R_{\alpha}^{\otimes(k-1)}(A_1\times \dots \times A_k)
	&= \int_{A_1} \int_{A_2} \dots \int_{A_k} R_{\alpha}(z_{k-1},\dint z_k) \dots R_{\alpha}(z_1,\dint z_2) \mathbb{P}_{Z_1}(\dint z_1)\\ 
	&= \mathbb{P}(Z_1\in A_1,\dots,Z_k\in A_k\mid\Theta=\alpha)
	\end{align*}
	for any $A_1,\dots,A_k \in \mathcal{B}([0,2\pi)^{3})$. For $k\geq2$ we exploit the fact that
	\[
	\mathbb{P}_{Z_1}\otimes R_\alpha^{\otimes(k-1)}(A_1\times \dots \times A_k)  
	= \int_{A_1\times \dots \times A_{k-1}} R_\alpha(z_{k-1},A_k)\; \mathbb{P}_{Z_1}\otimes R_\alpha^{\otimes(k-2)}(\dint z_1 \dots \dint z_{k-1}). 
	\]
	With that we add another auxiliary observation.
	\begin{lem} 
		Let $\alpha,\theta \in (0,2\pi]$ and 
		let		
		$g_\theta$, $\widetilde{g}_\theta$ be as in Lemma~\ref{lem: a_s_R_alpha}. Then, for any $k\in\N$ and  $C_1, \ldots, C_k \in \mathcal{B}([0,2\pi)^3)$ we have
		\begin{equation}\label{eq: gamma-chain under rotation}
		\mathbb{P}_{Z_1} \otimes R^{\otimes(k-1)}_{g_\theta^{-1}(\alpha)} \left(\widetilde{g}_\theta(C_1) \times \ldots \times \widetilde{g}_\theta(C_k)\right) 
		= 	\mathbb{P}_{Z_1} \otimes R^{\otimes(k-1)}_{\alpha}  \left(C_1\times \ldots \times C_k\right).
		\end{equation}
	\end{lem}
	\begin{proof}
		For $G_\alpha$ as defined in Lemma~\ref{lem: a_s_R_alpha} and 
		for any $k\in\N$
		we show by induction that
		\begin{equation}
		\label{eq: 1st_induc}
		\mathbb{P}_{Z_1} \otimes R^{\otimes(k-1)}_{\alpha}  \left(G_\alpha \times \ldots \times G_\alpha \right)=1.
		\end{equation}
		For $k=1$ we have
		\[
		\mathbb{P}_{Z_1}(G_\alpha) = \int_{[0,2\pi)} \mathbf{1}_{G_\alpha}(\gamma,\gamma,\gamma)\; \frac{\dint \gamma}{2\pi} = \mathcal{U}_{[0, 2\pi)}([0,2\pi)\setminus\{\alpha\}) =1.
		\]
		Assume \eqref{eq: 1st_induc} is true for $k$. Then, using from Lemma~\ref{lem: a_s_R_alpha} that $R_\alpha(z_k,G_\alpha)=1$, yields
		\begin{align*}
		\mathbb{P}_{Z_1} \otimes R^{\otimes k}_{\alpha}  \left(G_\alpha \times \ldots \times G_\alpha \right)
		& = \int_{G_\alpha\times \dots \times G_{\alpha}} R(z_{k},G_\alpha)\; \mathbb{P}_{Z_1}\otimes R^{\otimes(k-1)}(\dint z_1 \dots \dint z_{k})\\ 
		& =	\mathbb{P}_{Z_1} \otimes R^{\otimes (k-1)}_{\alpha}  \left(G_\alpha \times \ldots \times G_\alpha \right) =1.
		\end{align*}
		
		We prove now \eqref{eq: gamma-chain under rotation} again by induction.
		Indeed by the change of variables $\widetilde{\gamma} = g_\theta(\gamma)$ we have for all $C_1 \in \mathcal{B}([0,2\pi)^3)$ that
		\begin{align*}
		\mathbb{P}_{Z_1}\left(\widetilde{g}_\theta(C_1)\right) 
		&=  \int_0^{2\pi}\mathbf{1}_{C_1}\left(g_\theta(\gamma), g_\theta(\gamma), g_\theta(\gamma)\right)\,  \frac{\dint \gamma}{2\pi}
		=  \int_0^{2\pi}\mathbf{1}_{C_1}\left(\widetilde{\gamma}, \widetilde{\gamma}, \widetilde{\gamma}\right)\,  \frac{\dint \widetilde{\gamma}}{2\pi}
		= 	\mathbb{P}_{Z_1}\left(C_1\right),
		\end{align*}
		which gives the statement for $k=1$. Assume 
		that \eqref{eq: gamma-chain under rotation} holds for $k \in \mathbb{N}$. Then
		\begin{align*}
		& \mathbb{P}_{Z_1} \otimes R^{\otimes k}_{g_\theta^{-1}(\alpha)}  \left(\widetilde{g}_\theta(C_1) \times \ldots \times \widetilde{g}_\theta(C_{k+1})\right)\\
		= & \int_{\widetilde{g}_\theta(C_1) \times \ldots \times \widetilde{g}_\theta(C_{k})} R_{g_\theta^{-1}(\alpha)}\big(z_k, \widetilde{g}_\theta(C_{k+1})\big) \;
		\mathbb{P}_{Z_1} \otimes R^{\otimes(k-1)}_{g_\theta^{-1}(\alpha)}(\dint z_1\dots \dint z_{k})
		\\
		= &
		\int_{C_1 \times \ldots \times C_{k}} R_{g_\theta^{-1}(\alpha)}\big(\widetilde{g}_\theta(z_k), \widetilde{g}_\theta(C_{k+1})\big) 
		\;\mathbb{P}_{Z_1} \otimes R^{\otimes(k-1)}_{\alpha}(\dint z_1 \dots \dint z_k)\\
		\underset{\eqref{eq: 1st_induc}}{=} &
		\int_{(C_1 \times \ldots \times C_{k})\cap G_\alpha^k} R_{g_\theta^{-1}(\alpha)}\big(\widetilde{g}_\theta(z_k), \widetilde{g}_\theta(C_{k+1})\big) 
		\;\mathbb{P}_{Z_1} \otimes R^{\otimes(k-1)}_{\alpha}(\dint z_1 \dots \dint z_k)\\
		\underset{\eqref{eq: pushforward_coincides}}{=} &
		\int_{(C_1 \times \ldots \times C_{k})\cap G_\alpha^k} R_{\alpha}\big(z_k, C_{k+1}\big) 
		\;\mathbb{P}_{Z_1} \otimes R^{\otimes(k-1)}_{\alpha}(\dint z_1 \dots \dint z_k)
		\underset{\eqref{eq: 1st_induc}}{=} 
		\mathbb{P}_{Z_1} \otimes R^{\otimes k}_{\alpha}  \left(C_1\times \ldots \times C_{k+1}\right),
		\end{align*}
		which finishes the proof.
	\end{proof}
	Now we turn to the proof of Lemma~\ref{lem: pushforward_invariance}.
	\begin{proof}
		We apply \eqref{eq: gamma-chain under rotation} to the definition of $Q_{g_\theta^{-1}(S)}$ and get for all $B \in \mathcal{B}(S)$ that
		\begin{align*}
		&Q_{g_\theta^{-1}(S)}(g_\theta^{-1}(\alpha),g_\theta^{-1}(B)) \\
		&= \sum_{k= 1}^{\infty} \mathbb{P}\left[\Gamma_k \in g_\theta^{-1}(B) \cap g_\theta^{-1}(S), \Gamma_1 \in g_\theta^{-1}(S)^c, \ldots, \Gamma_{k-1} \in g_\theta^{-1}(S)^c \mid \Theta = g_\theta^{-1}(\alpha)\right]\\
		&= \sum_{k=1}^{\infty} \mathbb{P}_{Z_1} \otimes R^{\otimes(k-1)}_{g_\theta^{-1}(\alpha)}  \left(\widetilde{g}_\theta(S^c\times[0,2\pi)^2)^{k-1} \times \widetilde{g}_\theta(S\cap B \times[0,2\pi)^2 )\right) \\
		&
		\underset{\eqref{eq: gamma-chain under rotation} }{
			=} \sum_{k=1}^{\infty} \mathbb{P}_{Z_1} \otimes R^{
			\otimes(k-1)}_{\alpha}  \left((S^c\times[0,2\pi)^2) \times \ldots \times (S^c\times[0,2\pi)^2) \times (S\cap B \times[0,2\pi)^2 )\right) 
		= Q_S(\alpha,B),
		\end{align*}
		which verifies the claim.
\end{proof}

\end{appendix}
%
%

\begin{acks}[Acknowledgments]
 The authors thank Michael Habeck, Philip Sch\"ar and Bj\"orn Sprungk for comments on a preliminary version of the manuscript and fruitful discussions about this topic. Mareike Hasenpflug and Daniel Rudolf gratefully acknowledge support of the DFG within pro\-ject 432680300 -- SFB 1456 subproject B02.
\end{acks}
\bibliographystyle{imsart-nameyear} 


\end{document}